\documentclass[final]{siamltex}

\usepackage{subfigure}

\usepackage{graphicx}
\usepackage{amssymb,mathrsfs, amsmath}
\usepackage{graphicx}
\input epsf.tex
\usepackage{epsf}
\usepackage{epsfig}

\usepackage[vcentering,dvips]{geometry}
\newtheorem{remark}[theorem]{Remark}

\newcommand{\sgn}{\operatorname{sgn}}
\newcommand{\tr}{\operatorname{Tr}}
\newcommand{\A}{\mathscr{S}}
\newcommand{\M}{\mathscr{M}}
\newcommand{\PP}{\mathcal{P}}
\newcommand{\R}{\mathbb{R}}
\newcommand{\s}{\sigma}
\newcommand{\argmin}{\arg \min}
\newcommand{\J}{\mathcal J}

\usepackage[usenames]{color}
\usepackage{fullpage}

\def\nnew{\color{black}}
\def\bnew{\color{black}}

\begin{document}

\title{Low-rank matrix recovery via iteratively reweighted least squares minimization}
\author{
Massimo Fornasier\footnote{Johann Radon Institute for Computational and 
Applied Mathematics, Austrian Academy of Sciences, Altenbergerstrasse 69, A-4040 Linz, Austria,
email: {\tt massimo.fornasier@oeaw.ac.at}.}, \  
Holger Rauhut\footnote{
Hausdorff Center for Mathematics \& Institute for Numerical Simulation, University of Bonn, Endenicher Allee 60,
D-53115 Bonn, Germany,
email: {\tt  rauhut@hcm.uni-bonn.de}.}, \ and
Rachel Ward\footnote{Mathematics Department, Courant Institute of Mathematical Sciences, New York University, 251 Mercer Street, New York, N.Y. 10012-1185, U.S.A., email:  
{\tt rward@cims.nyu.edu}.}}
\date{\today}

\maketitle


\begin{abstract}
We present and analyze an efficient implementation of an iteratively reweighted least squares algorithm for recovering a matrix from
a small number of linear measurements. The algorithm is designed for the simultaneous promotion of both a minimal {\it nuclear norm} and an approximatively {\it low-rank} solution.  
Under the assumption that the linear measurements fulfill a suitable generalization of the {\it Null Space Property} known in the context of compressed sensing, the algorithm is guaranteed to recover iteratively any matrix with an error of the order of the best $k$-rank approximation. In certain relevant cases, for instance for the {\it matrix completion} problem, our version of this algorithm can take advantage of the {\it Woodbury matrix identity}, which allows to expedite the solution of the least squares problems required at each iteration. We present numerical experiments which confirm the robustness of the algorithm for the solution of matrix completion problems, and demonstrate its competitiveness with respect to other techniques proposed recently in the literature. 
\end{abstract}

\\


\section{Introduction}

Affine rank minimization refers to the problem of finding a matrix of minimal rank consistent with
a given underdetermined linear system of equations. This problem arises in many areas of 
science and technology, including system identification \cite{mofa10-1}, collaborative filtering, 
quantum state tomography \cite{beeiflgrli09,gr09-2}, 
signal processing, and image processing. An important
special case is the matrix completion problem \cite{care09,cata10,re09}, 
where one would like to fill in missing entries of a large data matrix which is assumed to have low-rank.  
  
Unfortunately, the affine rank minimization problem is
NP-hard in general \cite{fapareXX,na95}; therefore, it is desirable to have tractable alternatives.
{\bnew In \cite{fa02-2}, Fazel studied \emph{nuclear norm} minimization for this purpose, 
which was known to be a good proxy for rank
minimization.} The nuclear norm of a matrix is the $\ell_1$-norm of its singular values, and 
is the smallest convex envelope of the rank function \cite{fa02-2}.  Reformulated 
as a semidefinite program, nuclear norm minimization can be solved with
efficient methods \cite{bova04}. 

\subsection{Contribution of this paper}
Unfortunately, standard semidefinite programming tools work efficiently for solving
nuclear norm minimization problems
only for matrices up to size approximately $100 \times 100$.
Therefore, it is crucial to develop fast algorithms that are specialized to nuclear norm minimization
(or other heuristics for rank minimization).
 So far, several approaches have been suggested
 \cite{brle09,cacash10,goma09-1,kemooh09,kemooh09-1}. 
 Some aim at general nuclear norm
 minimization problems and others are specialized to matrix completion problems.
Inspired by the iteratively reweighted least-squares algorithm analyzed in \cite{dadefogu10} for sparse vector recovery, we develop a variant of this algorithm for nuclear norm minimization / low-rank matrix recovery. Each step of the proposed algorithm requires the computation
of a partial singular value decomposition and the solution of a (usually small) 
least squares problem, and both
of these tasks can be performed efficiently.   
{\bnew The analysis of our algorithm is based on a suitable matrix extension of the {\it null space
property}, well-known in the approximation theory literature in connection with $\ell_1$-minimization, see \cite{pink11,pink89} and reference therein, and recently popularized in the context of {\it compressive sensing} \cite{fora10-1}.} We show that the algorithm essentially has the same recovery guarantees as nuclear norm minimization.
Numerical experiments also show
that our algorithm is competitive with other state-of-the-art algorithms in the matrix completion setup
\cite{kemooh09,kemooh09-1, goma09-1}.  Unfortunately, the null space property fails in the matrix completion setup, and
the theoretical analysis of the algorithm seems to be much more involved. 
Such a theoretical analysis will  be postponed to later investigations.

\subsection{Low-rank matrix recovery and applications}
In the following, we will deal with real or complex matrices indifferently and we denote the space of $n \times p$ matrices by $M_{n \times p}$.
Given a linear map $\A: M_{n \times p} \to \mathbb{C}^m$, with $m \ll pn$, 
and a vector $\M \in \mathbb{C}^m$, 
we consider the affine rank minimization problem
\begin{equation}\label{rank:min}
\min_{X \in M_{n \times p}} \rank(X) \quad \mbox{ subject to } \A(X) = \M.
\end{equation}
An important special case of low-rank matrix recovery is matrix completion, where $\A$ samples entries, 
\begin{equation}\label{matrix:completion}
\A(X)_\ell = x_{ij},
\end{equation} 
for some $i,j$ depending on $\ell$.  Such low-rank matrix recovery problems often arise; examples include the {\it Netflix problem}\footnote{Netflix Prize sought to substantially improve the accuracy of predictions about how much someone is going to enjoy a movie based on their movie preferences {\tt http://www.netflixprize.com/}.} or the recovery of positions from partial distance information \cite{capl09}. We refer to \cite{care09,cata10} for further details.  
As an example of low-rank matrix recovery from more general linear measurements, we can turn to quantum state tomography \cite{beeiflgrli09}, where one tries to recover a quantum state,
that is, a square matrix $X \in M_{n \times n}$, that is positive semidefinite and has trace
$1$. A pure state has rank $1$, and a mixed state is of low-rank, or approximately low-rank.
Then one has given a collection of matrices
$A_{k} \in M_{n \times n}$, $k = 1,\hdots,n^2$, (so called Pauli-Matrices) 
and takes partial ``quantum observations'' $\M_j = \tr(A_{k_j}^* X)$, 
$j=1,\hdots,r$ with $r \leq n^2$, and the task is to recover a low-rank state $X$. 
We refer to \cite{beeiflgrli09,gr09-2} for details.

\subsection{Theoretical results}
\label{sec:theory}

As already mentioned the rank minimization problem \eqref{rank:min}
is NP-hard in general, and therefore we consider its convex relaxation
\begin{equation}\label{nucl:norm:min}
\min_{X \in M_{n \times p}} \|X\|_* \quad \mbox{ subject to } \A(X) = \M,
\end{equation}
where  $\|X\|_* = \sum_{i=1}^n \sigma_i(X)$ denotes the nuclear norm 
(or Schatten-$1$ norm, or trace norm), where $\sigma_i(X)$ are the singular values
of $X$.
There are two known regimes where 
nuclear norm minimization can be guaranteed to return minimal-rank solutions:

\paragraph{1.  RIP measurement maps}

The (rank) restricted isometry property (RIP) is analogous to the by-now classical restricted
isometry property (RIP) from compressed sensing \cite{capl09,fapareXX}.

\begin{definition}[Restricted Isometry Property \cite{fapareXX}]
\label{def:rip}
Let $\A: M_{n \times p} \rightarrow \mathbb{C}^m$ be a linear map.  For every integer $k$, 
with $1 \leq k \leq n$, define the $k$-restricted isometry constant $\delta_k = \delta_k(\A) > 0$ to be the smallest number such that   
\[
\label{ric}
(1-\delta_k) \|X\|_F^2 \leq \|\A(X)\|_{\ell_2^m}^2 \leq (1+\delta_k) \|X\|_F^2 
\]
holds for all $k$-rank matrices $X$.
\end{definition}

It is shown in \cite{fapareXX}
that nuclear norm miminization \eqref{nucl:norm:min} recovers all matrices $X$ 
of rank at most $k$ from the measurements $\M=\A(X)$ {\nnew{provided $\delta_{5k}$ is small enough. Below we improve this to 
$\delta_{4k} \leq \sqrt{2} - 1$, and we refer to \cite{capl09,oymofaha11} for related results}}.


The restricted isometry property is known to hold for 
Gaussian (or more generally subgaussian) measurement maps $\A$ in \cite{fapareXX,capl09}. These are maps of the form
\begin{equation}
\label{eq:formM}
\A(X)_\ell = \sum_{k,j} a_{\ell,k,j} X_{k,j}, \quad \ell = 1,\hdots,m, 
\end{equation}
where the $a_{\ell,k,j}$ are independent normal distributed random variables with mean
zero and variance $1/m$. Such a map satisfies $\delta_k \leq \delta \in (0,1)$
with high probability provided
\begin{equation}\label{bound:r}
m \geq C_\delta \max\{p,n\} k,
\end{equation}
see Theorem 2.3 in \cite{capl09}.
Since the degrees of freedom of a rank $k$ matrix $X \in M_{n \times p}$ are
$k(n+p-k)$, the above bound matches this number up to possibly a constant. 
Therefore, the bound \eqref{bound:r} is optimal. 

Using recent results in \cite{ailib, kw-10}, it follows that the restricted isometry property also holds for certain \emph{structured} random matrices if slightly more measurements are allowed;  
in particular, for maps $\A$ of the form \eqref{eq:formM} where the $a_{\ell,k,j}$ -- unraveled for each $\ell$ into vectors of length $pn$ -- correspond to $m$ rows randomly selected from the $pn \times pn$ discrete Fourier matrix (or 2D Fourier transform matrix) with randomized column signs, 
$\delta_k \leq \delta \in (0,1)$ with high probability as long as 
\begin{equation}\label{fourier:r}
m \geq C_{\delta} \max\{p,n\} k \log^{\nnew{4}}(pn).
\end{equation}
This follows from the recent findings in \cite{ailib, kw-10} that such random partial Fourier measurements satisfy a concentration inequality of the form, for all $0 < \varepsilon < 1$,
\begin{equation}
\label{eq:concentration}
\mathbb{P}\Big(\left| \| \A(X) \|^2 - \| X \|_F^2 \right| \geq \varepsilon \| X \|_F^2 \Big) \leq 2 \exp{(-\frac{m}{2}C_\varepsilon \log^{-4}(pn))}.
\end{equation}
Subgaussian measurement maps also satisfy \eqref{eq:concentration}, and for these maps, the factor of $\log^{-4}(pn)$ in the exponent can be removed.  The proof of RIP for subgaussian random ensembles appearing in {\nnew{Theorem 2.3 in \cite{capl09} (see also Theorem $4.2$ in \cite{fapareXX})}} 
relies only on such concentration for subgaussian measurement maps. 
Therefore, obvious modifications to that proof to accommodate the additional logarithmic factors in the exponent \eqref{eq:concentration} give the stated RIP results for random partial Fourier measurements.  

\paragraph{2. Matrix completion}In the matrix completion setup \eqref{matrix:completion} where measurements are pointwise observations of entries of the matrix, there are obvious 
rank one matrices in the kernel of the operator $\A$; therefore, the RIP fails completely in this setting, and  `localized' low-rank matrices in the null space of $\A$ cannot be recovered by any method whatsoever.
 However, if certain conditions on the left and right singular vectors of the underlying low-rank matrix are imposed, essentially requiring that such vectors are uncorrelated with
the canonical basis, then it was shown in \cite{care09,cata10,re09} that such incoherent matrices of rank at most $k$ can be recovered from $m$
randomly chosen entries with high probability provided 
\[
m \geq Ck \max\{n,p\} \log^2(\max\{n,p\}).
\]
Up to perhaps the exponent $2$ at the $\log$-term, this is optimal. We refer to \cite{care09,cata10,re09} for detailed statements. In \cite{beeiflgrli09,gr09-2} extensions
to quantum state tomography can be found.

\section{Notation and Preliminaries}

\subsection{Notation}

The entries of a matrix $X \in M_{n \times p}$ are denoted by lower case letters and the corresponding indices, i.e., $X_{ij}=x_{ij}$. We denote the adjoint matrix $X^* \in M_{p \times n}$ as usual.
In the following and without loss of generality, we assume that $n \leq p$.
For a generic matrix $X \in   M_{n \times p}$ we write its {\it singular value decomposition} 
$$
 X= U \Sigma V^*
$$
for unitary matrices $U \in M_{n \times n}, V \in M_{p \times p}$  and a diagonal matrix
$\Sigma = \diag(\s_1,\dots,\s_n) \in M_{n \times p}$, 
where $\s_1 \geq \s_2\geq \dots \s_n \geq 0$ are the {\it singular values}. In the following, we denote $\sigma(X)$ the vector of the singular values of the matrix $X$. 
For a specific matrix $X$ we indicate the singular value decomposition as $X= U_X \Sigma_X V_X^*$.
For self-adjoint matrices $X = X^*$, we have $V=U$. We write $X \succ 0$ if $X$ is self-adjoint and positive-definite.
In this case we can define, for $\alpha \in \R$, the $\alpha-$th 
power of the matrix $X$ 
by $X^\alpha = U \Sigma^\alpha U^*$, 
where $\Sigma^\alpha = \diag(\s_1^\alpha,\dots,\s_n^\alpha)$.
If $X$ is positive semi-definite, we write $X \succeq 0$.

The rank of $ X \in M_{n \times p}$ denoted by $\rank(X)$ equals the number of nonzero singular values of $X$. We will say that $X$ is a \emph{$k$-rank matrix} if $\rank(X) \leq k$.
The trace of a square matrix $X \in M_{n \times n}$ is the sum of its diagonal entries, i.e.,
$\tr(X) = \sum_{i=1}^n x_{ii}$. 
The trace satisfies $\tr(X) = \overline{\tr(X^*)}$ and is cyclic, that is, $\tr(XY) = \tr(YX)$ for all 
matrices $X,Y$ with matching dimensions. If $\lambda_i$ are the eigenvalues of $X$ then $\tr(X) = \sum_{i} \lambda_i$.
{\nnew{Denote by $E_{ij}$ the matrix whose $(i,j)$-entry is $1$ and all other entries take the value $0$.}}
Then $\tr(E_{ij} X) = x_{ji}$. 

The space of complex matrices $M_{n\times p}$ forms a Hilbert space 
when endowed with the natural scalar product $\langle X, Y \rangle = \tr (X Y^*),$
which induces the {\it Frobenius norm} $\| X\|_F = \langle X, X \rangle^{1/2}.$
The following identities hold: $
\| X\|_F =\left ( \sum_{i=1}^n \sum_{j=1}^p |x_{ij}|^2 \right )^{1/2} = \left ( \sum_{i=1}^n \s_i^2 \right)^{1/2}.$ 
More generally, we consider the Schatten $q$-norms $\|X\|_{S_q} := \|\sigma(X)\|_{\ell_q^n}$ \cite{bh97,hojo90}, which are based on the $\ell_q^n$-vector norms $\| x \|_{\ell_q^n} := \big( \sum_{i=1}^n |x_i|^q \big)^{1/q}, $ $1 \leq q < \infty$, and $\| x \|_{\ell_{\infty}^n} := \max_{i=1}^n |x_i|$. Of particular importance is the Schatten $1$-norm, or nuclear norm, which we also denote by
$\| X\|_* = \| X\|_{S_1} = \sum_{i=1}^n \s_i(X).$ The operator norm 
$\|X\|_{op} = \|X\|_{S_\infty}$ will be needed as well.

If $X X^*$ is invertible, then the {\it absolute value} of $X$ is defined by
\begin{equation}
\label{absval}
| X | = (X X^*)^{-1/2} X X^* {\bnew = (X X^*)^{1/2}};
\end{equation}
in this case, the Schatten $q$-norm satisfies
\begin{equation}
\label{schatten}
\| X\|_{S_q} := \left ( \tr |X|^q \right )^{1/q}
\end{equation}
for all $1 \leq q < \infty$. 

We consider linear operators on matrices of the type $\A: M_{n \times p} \to \mathbb C^m$. We denote the action of $\A$ on the matrix $X$ by $\A(X)$ in order to distinguish it from $\A X$ which may create confusion and be interpreted as a matrix-matrix multiplication in certain passages below.  Rather, as the vector space $M_{n \times p}$ is isomorphic to the vector space $\mathbb{C}^{n \times p}$, linear operators of the form $\A: M_{n \times p} \to \mathbb C^m$ may be interpreted as elements of $M_{m \times np}$.  We denote by $\A^*$ the adjoint operator of $\A$, such that
{\nnew{$
\langle \A(X), \lambda \rangle_{\mathbb C^m} = \langle X, \A^*(\lambda) \rangle$ for all $X \in M_{n \times p}, \lambda \in \mathbb C^m,
$}}
where the former scalar product is the Euclidean one on $\mathbb C^m$ and the latter scalar product is the one inducing the Frobenius norm. We denote by $I:=I_n \in M_{n \times n}$ the identity matrix. 

Finally, let us define the {\it $k$-spectral truncation} of $X$ by
$$
X_{[k]} = U \Sigma_{[k]} V^*,
$$
where $\Sigma_{[k]} = \diag(\sigma_1, \dots, \sigma_k,0, \dots,0)$. {\bnew Thus, this operation acts by setting to $0$ all the singular values with indexes from $k+1$ to $n$}.  We also introduce the $\varepsilon$-stabilization of $X$ by
\begin{equation}
\label{estab}
X_{\varepsilon} = U \Sigma_{\varepsilon} V^*,
\end{equation}
where 
$\Sigma_{\varepsilon} = \diag( \max\{ \sigma_j, \varepsilon \} )$. That is, under this operation, all singular values below a certain minimal threshold are increased to the threshold.  

For a self-adjoint matrix $X = U \Sigma U^* \in M_{n \times n}$, note that
\begin{eqnarray}
X_{\varepsilon} &=& U \widetilde{\Sigma}U^* +  \varepsilon I, 
\end{eqnarray}
where $\widetilde{\Sigma} =  \diag{\big( \max\{\sigma_j - \varepsilon, 0 \} \big)}$.  In particular,  if a significant number of the singular values of $X$ fall below the threshold $\varepsilon$, then $X_{\varepsilon}$ is decomposed into the sum of a low-rank matrix and {\nnew{a scaled identity matrix}}.

\subsection{Motivation and formulation of the algorithm}

Let us fix a sampling operator $\A: M_{n \times p} \to \mathbb C^m$ 
and a measured data vector $\M \in \mathbb C^m$. We are interested in the rank minimization problem
\begin{equation}
\label{rankmin}
\argmin_{\A (X) = \M} \rank (X).
\end{equation}
As mentioned in the introduction, this problem can be efficiently solved under suitable assumptions on $\A$ by considering its convex relaxation
\begin{equation}
\label{nuclmin}
\argmin_{\A (X) = \M} \| X \|_*,
\end{equation}
where rank minimization is substituted by nuclear norm minimization. 

We propose an algorithm for solving \eqref{nuclmin} reminiscent of iteratively reweighted least
squares algorithm for linearly-constrained $\ell_1$-norm minimization, and is based on the following
motivation: if all of the singular values of $X \in M_{n \times p}$ are nonzero, then, according to
\eqref{absval}-\eqref{schatten}, we may re-write its nuclear norm as
\begin{eqnarray}
 \| X\|_* = \| X\|_{S^1} &=& \tr \left [ (X X^*)^{-1/2} (X X^*) \right] = \| W^{1/2}X\|^2_{F},
 \end{eqnarray}
 where $W = (X X^*)^{-1/2}$.  This suggests the following approach for solving the nuclear norm minimization problem \eqref{nuclmin}: let us initialize a weight matrix $W^0 \in M_{n \times n}$, and then recursively define, for $\ell = 0,1,\dots,$
\begin{equation}
\label{xn}
X^{\ell+1} = \arg \min_{\A(X)=\M } \|(W^{\ell})^{1/2} X  \|_F^2, \hspace{5mm} \textrm{and } \hspace{3mm} W^{\ell+1} = \Big( X^{\ell+1} \big( X^{\ell+1} \big)^* \Big)^{-1/2}.
\end{equation}

Observe that the minimization $ \bar{X} = \arg \min_{\A(X) = \M} \| W^{1/2} X \|_F^2 $ can be reformulated as a weighted least squares problem with $m$ linear constraints on the $n p$ variables $x_{ij}$; each of the updates for $X^{\ell+1}$ and $W^{\ell+1}$ can then be performed explicitly. However,  once any of the singular values of $X^{\ell + 1}$ become small -- a desirable situation as we seek to recover low-rank matrices -- the computation of $W^{\ell+1}$  becomes ill-conditioned.   

To stabilize the algorithm, we fix a parameter $\varepsilon > 0$, and {\nnew{increase}} 
the small singular values of $X^{\ell}$ according to $\sigma_i(X^{\ell}) \longrightarrow \max\{ \sigma_i(X^{\ell}), \varepsilon\}$ after updating $X^{\ell}$ but before computing $W^{\ell}$.   That is, we replace $X^{\ell}$ by its $\varepsilon$-stabilization $X^{\ell}_{\varepsilon} = U \Sigma^{\ell}_{\varepsilon} V^*$, as defined in \eqref{estab}.  The idea is that $X^{\ell}_{\varepsilon}$ is well-conditioned and, additionally, if $X^{\ell}$ is well-approximated by a $k$-rank matrix, then $X^{\ell}_{\varepsilon}$ is well-approximated by the same $k$-rank matrix.  Keeping $\varepsilon > 0$ fixed throughout the iterations, we would no longer expect the algorithm to converge to the nuclear norm solution \eqref{nuclmin}.  Nevertheless, if we allow for $\varepsilon = \varepsilon_{\ell} \rightarrow 0$ in such a way that $\varepsilon_{\ell} < \sigma_k(X^{\ell})$ is maintained at each iteration, then one may hope for both stability and convergence towards a $k$-rank solution.  Collecting these ideas together, we arrive at the following \emph{iterative reweighted least squares algorithm} for low-rank matrix recovery (IRLS-M):

\medskip

\noindent
\fbox{
\begin{minipage}{14cm}
{\bf IRLS-M algorithm for low-rank matrix recovery: } Initialize $W^0 = I \in M_{n \times n}$.  Set $\varepsilon_0:=1$, $K\geq k$, and $\gamma >0$. Then recursively define, for $\ell = 1,2,\dots,$
{\nnew{
\begin{align}
\label{xn2}
X^{\ell} &:=\arg \min_{\A(X)=\M } \| (W^{\ell-1})^{1/2} X \|_F^2,\\
\label{en}
\varepsilon_{\ell} &:=\min \left\{\varepsilon_{\ell-1}, \gamma \s_{K+1}(X^{\ell}) \right\}.
\end{align}}}
Compute $U^{\ell} \in M_{n \times n}$ and $\Sigma^{\ell} = \diag{(\sigma_1^{\ell}, ..., \sigma_n^{\ell} )} \in M_{n \times n}$ for which
{\nnew{
\begin{align}
\label{x:decomp}
X^{\ell} (X^{\ell})^* &= U^\ell \big( \Sigma^{\ell} \big)^2 (U^\ell)^*.\\
\label{wn}
{\mbox{Set }}\qquad\qquad \qquad W^{\ell} &= U^\ell \big( \Sigma_{\varepsilon_{\ell}}^{\ell} \big)^{-1} (U^\ell)^*,
\end{align}}}
where $\Sigma_{\varepsilon}$ denotes the $\varepsilon$-stabilization of the matrix $\Sigma$, 
see \eqref{estab}.
The algorithm stops if $\varepsilon_\ell=0$; in this case, define $X^j:=X^\ell$ for 
$j>\ell$.  In general, the algorithm generates an infinite sequence 
$(X^\ell)_{\ell\in \mathbb N}$ of matrices. 
\end{minipage}
}

\medskip

\paragraph{Reduction to iteratively reweighted least squares for vector recovery} Suppose now that  $n=p$ and the linear constraints $\A(X) = \M$ are equivalent to the $m \times n$ linear system $S x= \M$ exclusively acting on the vector $x = \diag{X}$, and $x_{ij}=0$ is enforced  for $i \neq j$. Then the nuclear norm minimization problem \eqref{nuclmin} reduces to that of finding the vector $x \in \mathbb{C}^n$ of minimal $\ell_1^n$-norm which solves the optimization problem
\begin{equation}
\label{l1min}
\argmin_{S z = \M} \| z \|_{\ell_1^n}.
\end{equation}
In this context, the updates for $X^{\ell+1}$ and $W^{\ell+1}$ in the IRLS-M algorithm reduce to vector updates $x^{\ell+1} \in \mathbb{C}^n$ and $w^{\ell+1} \in \mathbb{R}^n_+$.





Although iteratively reweighted least squares algorithms for solving constrained $\ell_1^n$ minimization problems of the form \eqref{l1min} 
have been around for half a century, it was not until recently \cite{dadefogu10} that conditions on the matrix $S\in M_{m \times n}$ were provided under which an {\nnew{iteratively}} reweighted least squares algorithm could be \emph{proven} to 
converge to the solution of the $\ell_1^n$ minimization problem \eqref{l1min}.   
In the algorithm proposed in \cite{dadefogu10}, the weight vector $w^{\ell}$ is updated slightly differently than in the current setting; in \cite{dadefogu10} 
it is updated according to $w_i^{\ell} = (|x_i^{\ell}|^2+ \varepsilon_{\ell}^2)^{-1/2},$
while our algorithm uses the update rule
\[
\label{update_rule}
w_i^{\ell}= \left\{ \begin{array}{cc}  |x^{\ell}_i|^{-1}, & | x_i^{\ell} | \geq \varepsilon_{\ell},  \\
\varepsilon_{\ell}^{-1}, & \textrm{else}.
\end{array} \right.
\]
Despite this difference, the assumptions, which we shall require on the measurements $\A$ in order to prove convergence 
of the IRLS-M algorithm to the nuclear-norm solution \eqref{nuclmin}, draw upon the the variational framework introduced in \cite{dadefogu10}.  
However, our update rule for the weights \eqref{update_rule} is reminiscent of a previously studied IRLS algorithm \cite{chamlions97}, 
which was introduced in the 
slightly different context of {\it total variation minimization} as opposed to $\ell_1$ minimization.  
However, in that algorithm, the parameter $\varepsilon_{\ell} = \varepsilon>0$ is fixed throughout the iteration; the corresponding iterations
$x^{\ell}_{\varepsilon}$ are shown to converge to a solution $x = x(\varepsilon)$, and one can show that $x(\varepsilon)$  
converges to the $\ell_1$-norm solution $x^*$ as $\varepsilon \rightarrow 0$ provided that $x^*$ is unique.  
For an overview of {\nnew{iteratively}} reweighted least squares minimization, we refer the reader to \cite{dadefogu10}.   

\subsection{Convergence results}

As a consequence of the main result of this paper, Theorem \ref{secondres}, we have the following convergence result for the IRLS-algorithm under the assumption that the measurement map has the restricted isometry property in Definition \ref{def:rip}.

\begin{proposition}
\label{prop:mini}
Consider the IRLS-M algorithm with parameters $\gamma = 1/n$ and $K \in  \mathbb{N}$.  
Let $\A: M_{n \times p} \rightarrow \mathbb{C}^m$ be a surjective map with restricted isometry constants $\delta_{3K}$, $\delta_{4K}$ satisfying $\eta =\frac{\sqrt{2}\delta_{4K}}{1-\delta_{3K}} < 1 - \frac{2}{K-2}$.  Then, if there exists a $k$-rank matrix $X$ satisfying $\A (X) = \M$ with $k < K - \frac{2\eta}{1 - \eta}$, the sequence $(X^{\ell})_{\ell \in \mathbb{N}}$ converges to $X$.
\end{proposition}

Actually, Theorem \ref{secondres} proves something stronger: the IRLS-M algorithm is \emph{robust}, in the sense that under the same conditions on the measurement map $\A$, the accumulation points of the IRLS-M algorithm are guaranteed to approximate an arbitrary $X \in M_{n \times p}$ from the measurements $\M = \A(X)$ to within a factor of the \emph{best $k$-rank approximation error} of $X$ in the nuclear norm.   

The IRLS-M algorithm also has nice features independent of the affine measurements.  In particular, the parameter $\varepsilon_\ell$ can serve as an {\it a posteriori} check on whether or not we have converged to a low-rank (or approximately low-rank) matrix consistent with the observations:  
\begin{remark}
\label{zerores}
If $\varepsilon_{\ell-1} > \varepsilon$ and $\varepsilon_{\ell} =  \gamma \sigma_{K+1}(X^\ell)=\varepsilon$, then 
{\nnew{$\| X^{\ell} - X^{\ell}_{[K]} \|_{op} \leq \varepsilon/\gamma.$}}
In particular, if $\varepsilon_{\ell} = 0$, then $X^{\ell}$ is a $K$-rank matrix.
\end{remark}


\section{Numerical Experiments}

In the matrix completion setup \eqref{matrix:completion} where the affine measurements are entries of the underlying matrix, the restricted isometry property fails, and so Proposition \ref{prop:mini} does not apply.  Nevertheless, we illustrate in this section numerical experiments which show
that the IRLS-M algorithm still works very well in this setting for recovering low-rank matrices, and is competitive with state-of-the-art algorithms in the matrix completion setup; namely, the Fixed Point Continuation algorithm (FPCA), as introduced in \cite{goma09-1}, and the OptSpace algorithm, as introduced in \cite{kemooh09}.   For further numerical experiments comparing an algorithm closely related to IRLS-M with other matrix completion algorithms such as Singular Value Thresholding \cite{cacash10}, we refer the reader to \cite{mofa10}.  

\begin{figure}[h]
\label{fig:1}
\subfigure[Relative error (Frobenius norm)]{
\includegraphics[scale=.52]{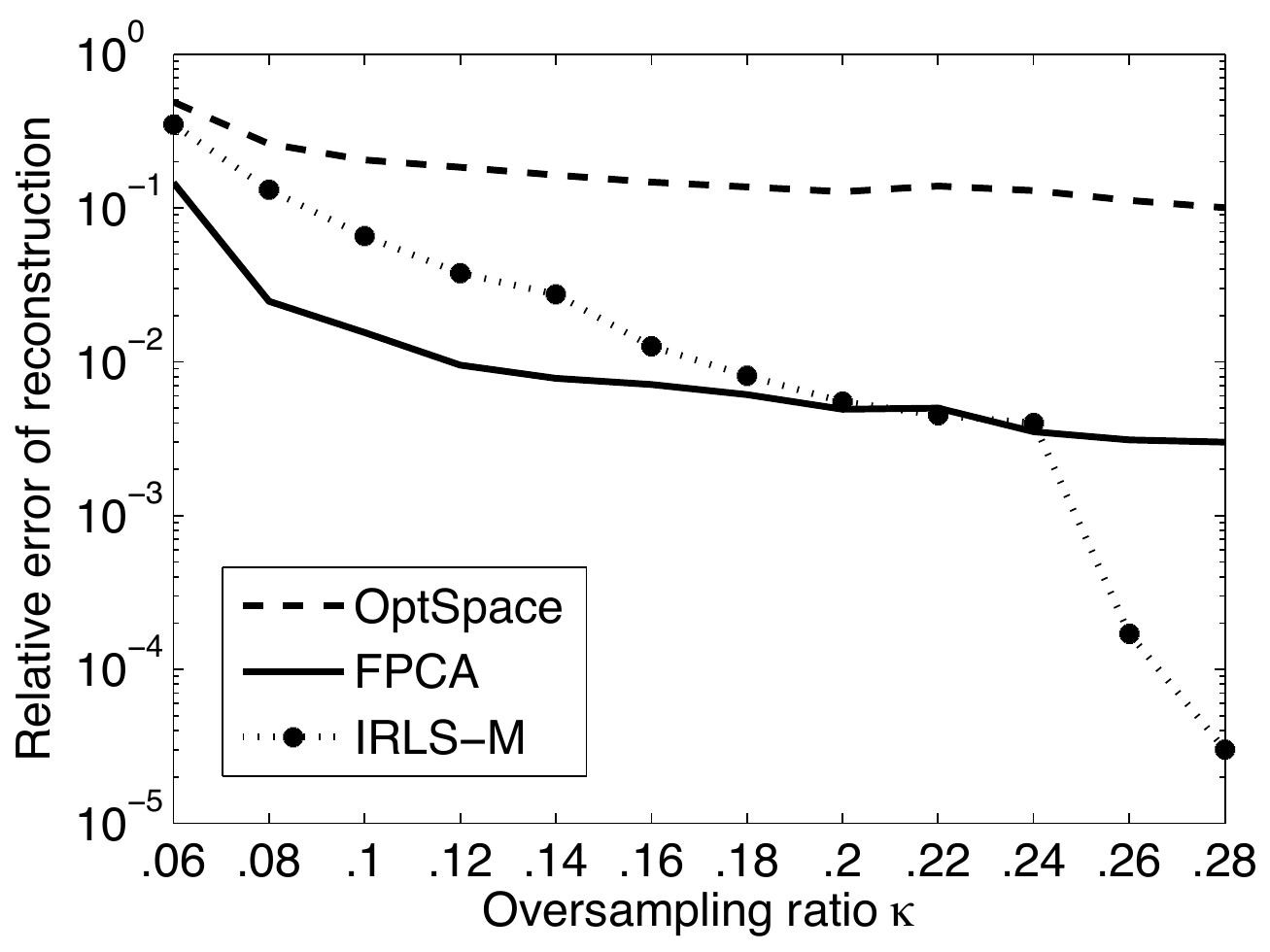} \label{1:a} 
}
\subfigure[Reconstruction time (s)]
{\includegraphics[scale =.52] {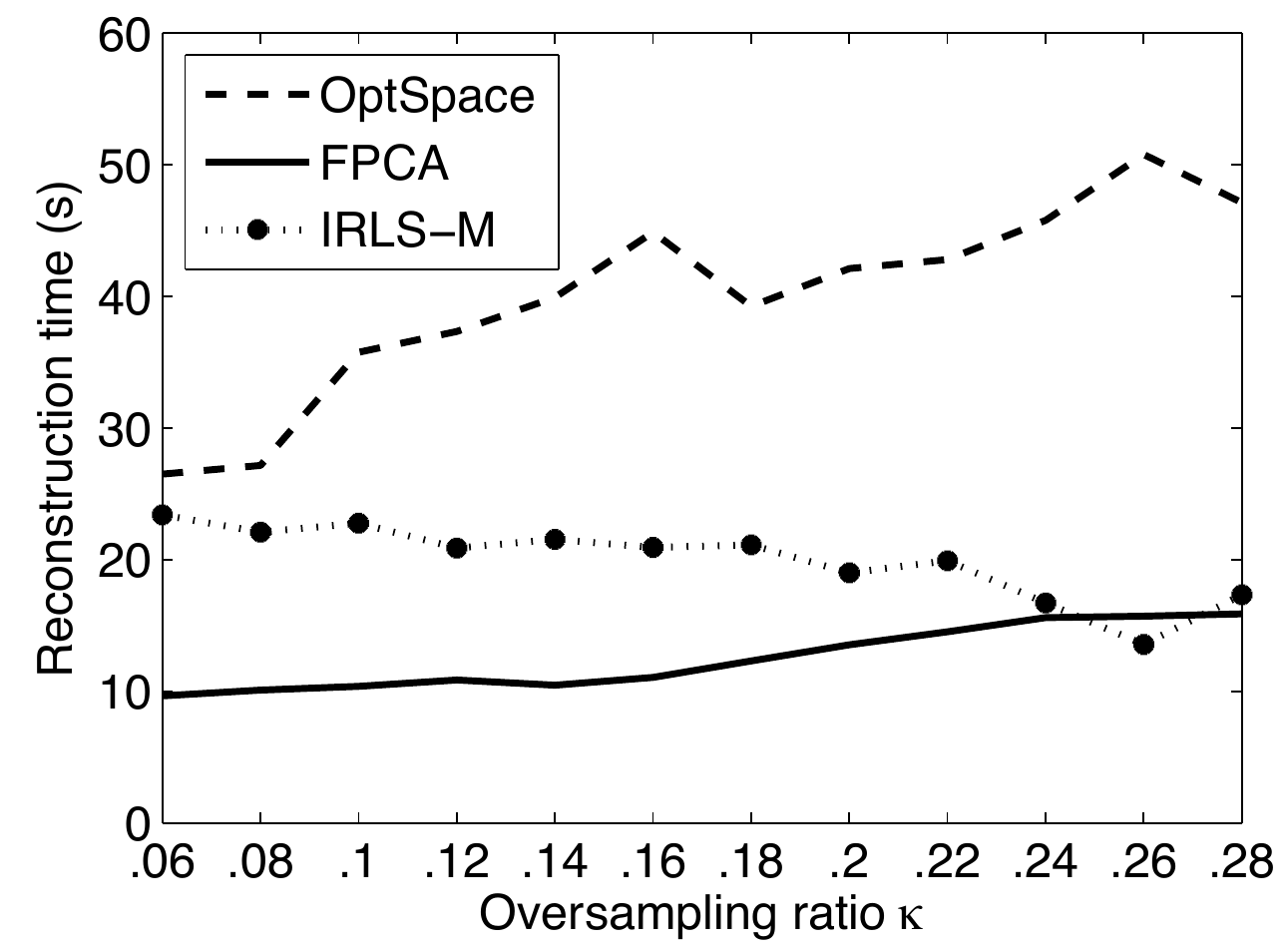} \label{1:c} 
}
\caption{Comparison of IRLS-M, FPCA, and OptSpace over a class of randomly-generated rank-$10$ matrices.}
\medskip
\end{figure} 

\begin{figure}[h]
\label{fig:2}
\subfigure[Relative error (with respect to Frobenius norm)]{
\includegraphics[scale=.52]{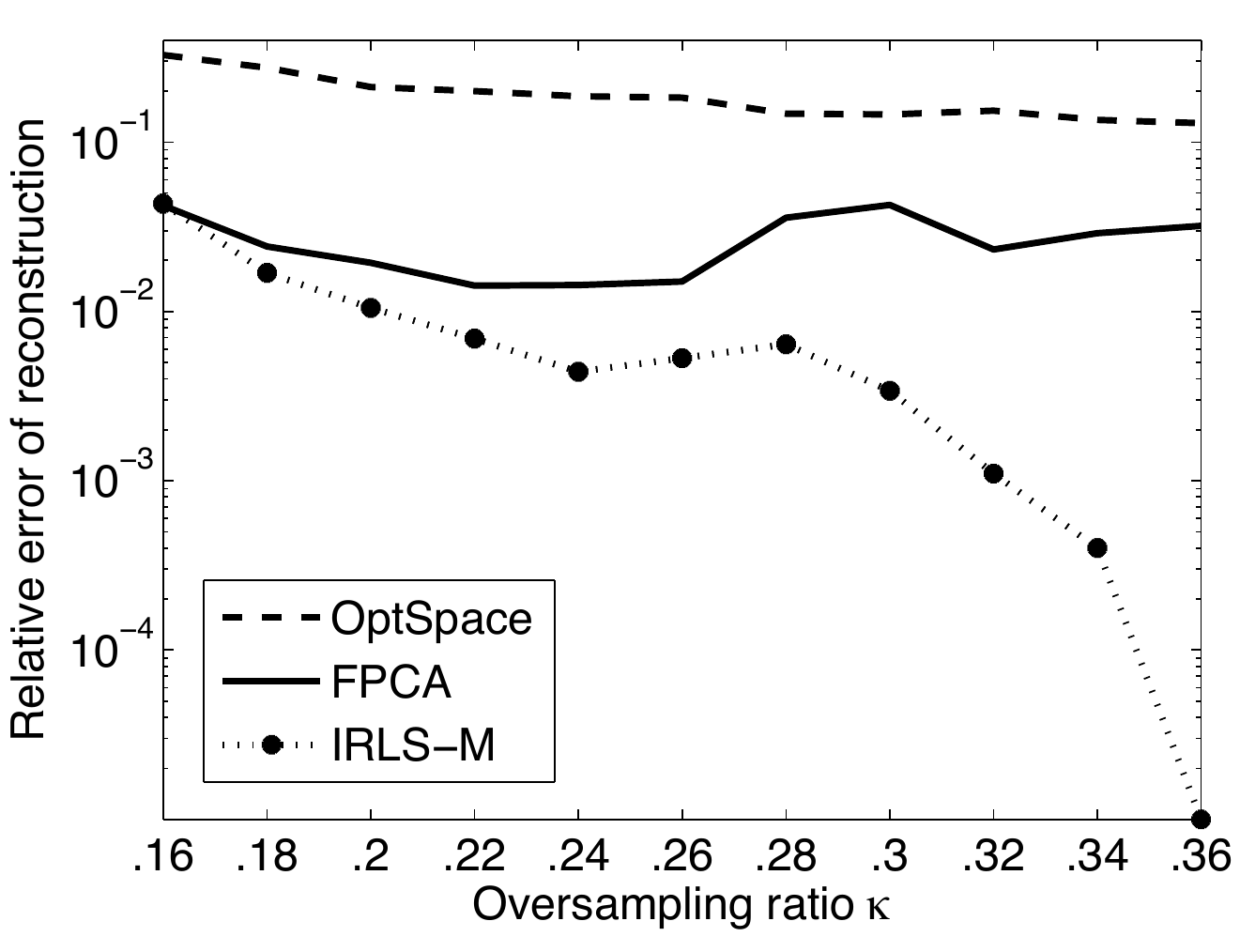} \label{2:a} 
}
\subfigure[Reconstruction time (s)]
{\includegraphics[scale =.52] {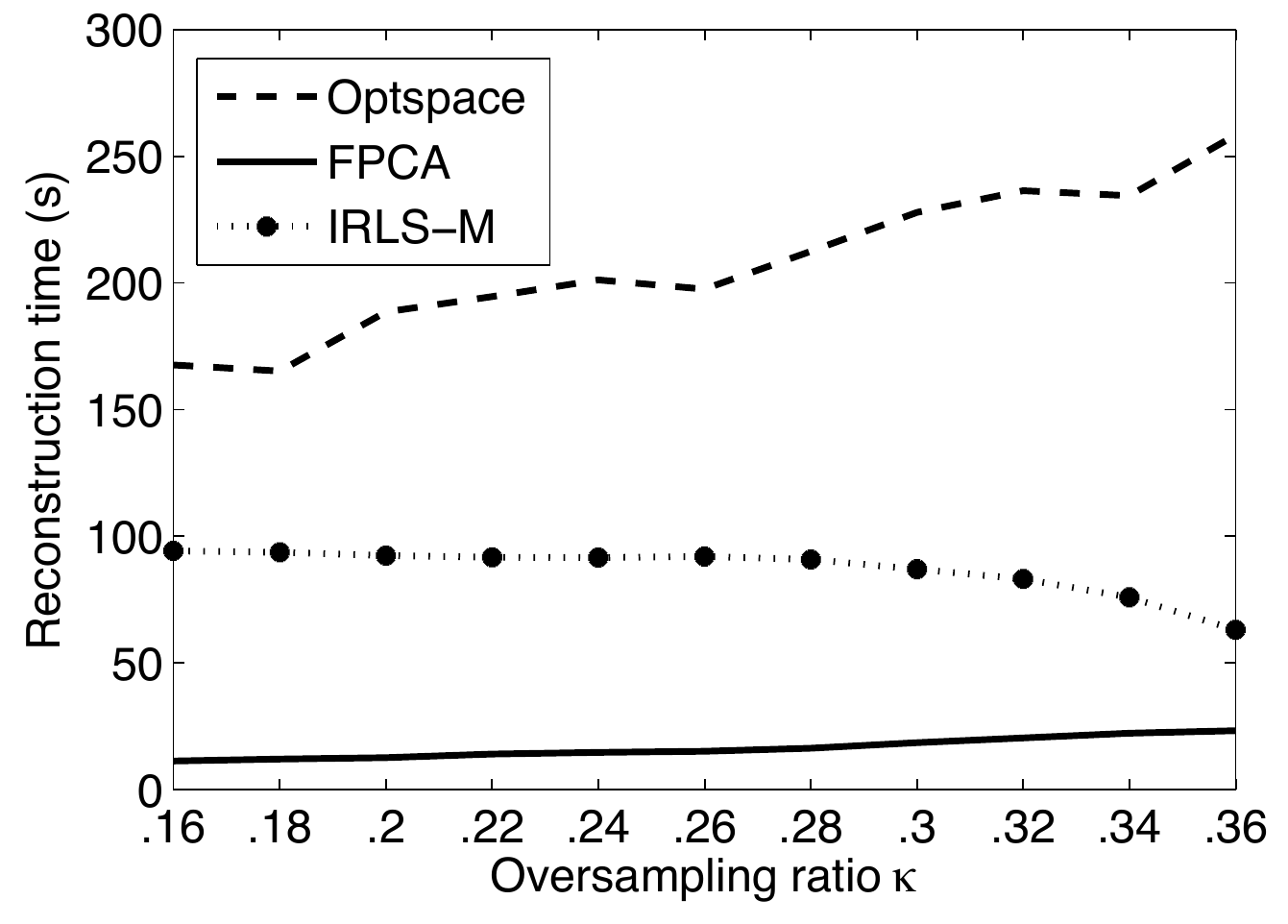} \label{2:c} 
}
\caption{Comparison of IRLS-M, FPCA, and OptSpace over a class of randomly-generated rank-$30$ matrices.}
\medskip
\end{figure}


\paragraph{1} In Figure \ref{fig:1}$.1$, we generate rank $k= 10$ matrices $X$ of dimensions $n = p = 500$, and of the form $X = U \Sigma V^*$, where the rows of $U \in M_{k \times n}$ and of $V \in M_{k \times n}$ are drawn {\bnew independently} from the multivariate standard normal distribution, and the elements of the diagonal matrix $\Sigma$ are chosen independently from a standard normal distribution (note that $U$, $\Sigma$, and $V$ do not comprise a singular value decomposition, although this process does generate matrices of rank $k=10$).  We plot out of $50$ trials (a) the average relative error $\| X - \widetilde{X} \|_{F} / \| X \|_{F}$ and (b) the average running time in reconstructing such matrices from $m = \kappa n p $ randomly chosen observed entries (matrix completion).  To produce Figure \ref{fig:2}$.2$, we repeat the same procedure, except with the rank of the underlying matrices now fixed to be $k=30$.

\medskip

\paragraph{2} Figures \ref{fig:3}$.3$ and \ref{fig:4}$.4$ are analogs of Figures \ref{fig:1}$.1$ and \ref{fig:2}$.2$ respectively, generated by matrices $X$ of rank $k = 10$ and $30$ and of dimension $n=p = 500$ using the same model.  But now the observed entries $\bar{X}_{i,j} = X_{i,j} + .1 {\cal N}_{i,j}$ are corrupted by additive noise; the ${\cal N}_{i,j}$ are independent and identically distributed standard normal random variables.  

\begin{figure}[h]
\label{fig:3}

\subfigure[Relative error (with respect to Frobenius norm)]{
\includegraphics[scale=.52]{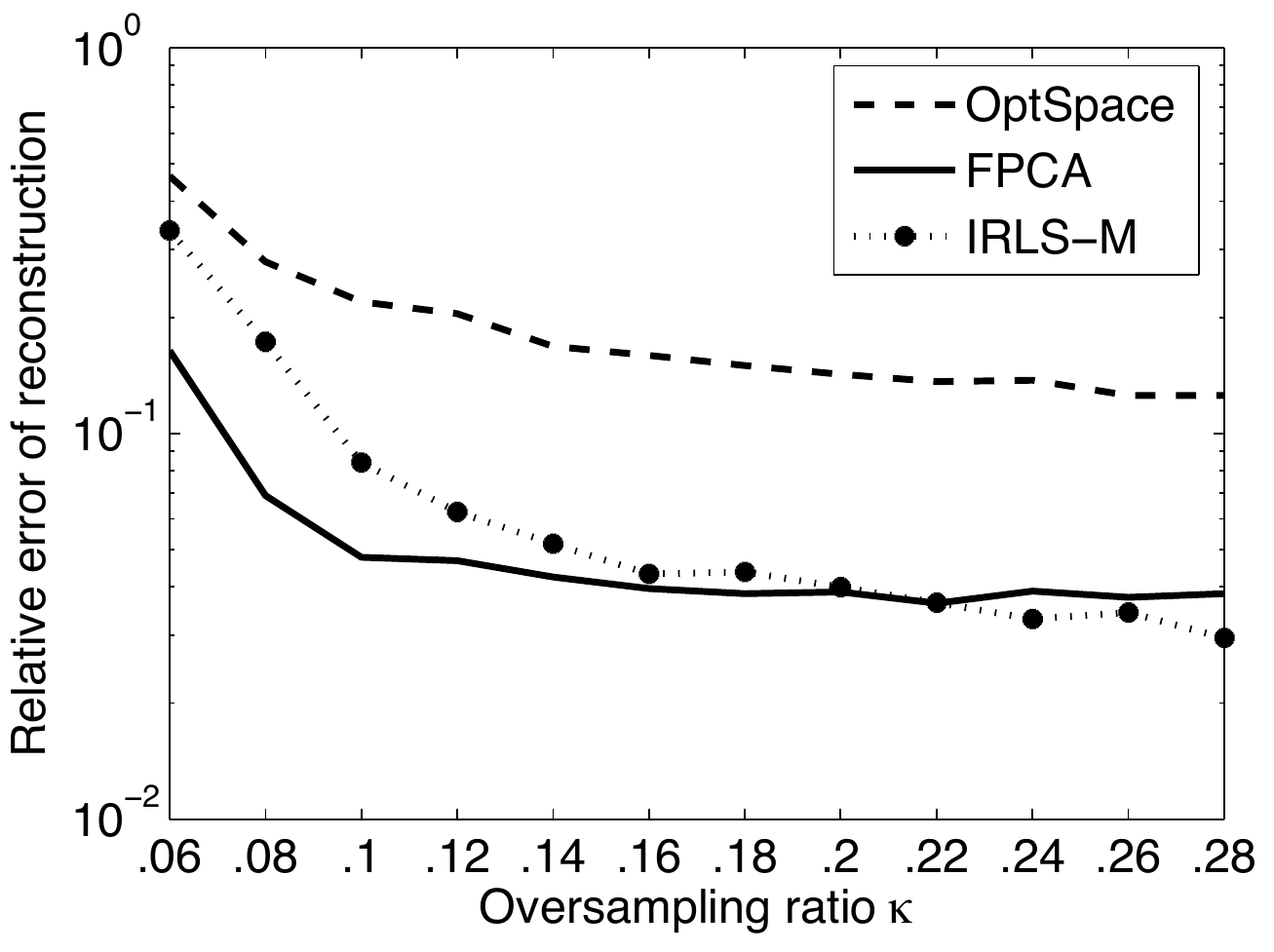} \label{3:a} 
}
\subfigure[Reconstruction time (s)]{ 
\includegraphics[scale =.52] {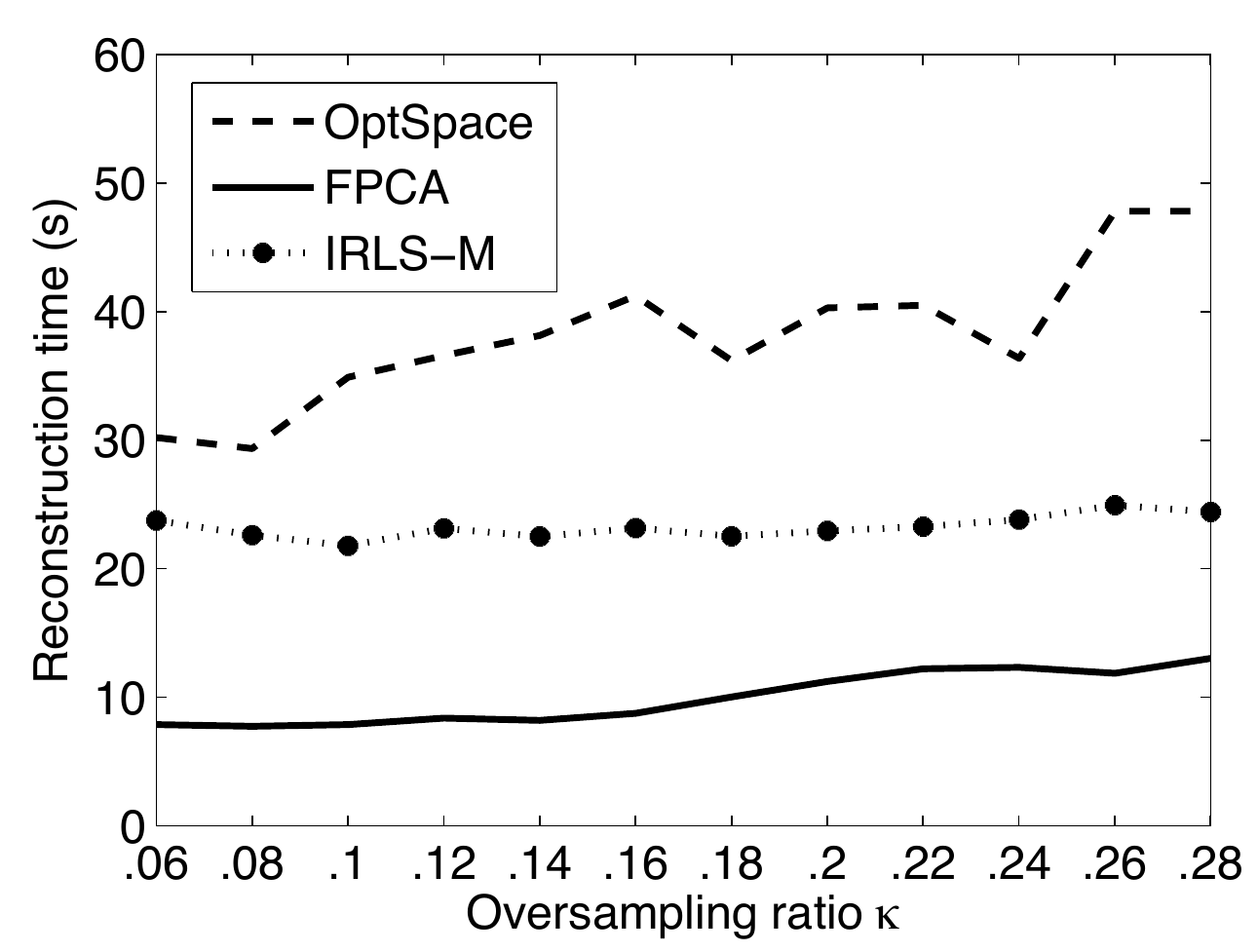} \label{3:b}
}

\caption{Comparison of IRLS-M, FPCA, and OptSpace over a class of randomly-generated rank-$10$ matrices subject to measurement noise.}
\medskip
\end{figure} 

\begin{figure}[h]
\label{fig:4}

\subfigure[Relative error (with respect to Frobenius norm)]{
\includegraphics[scale=.52]{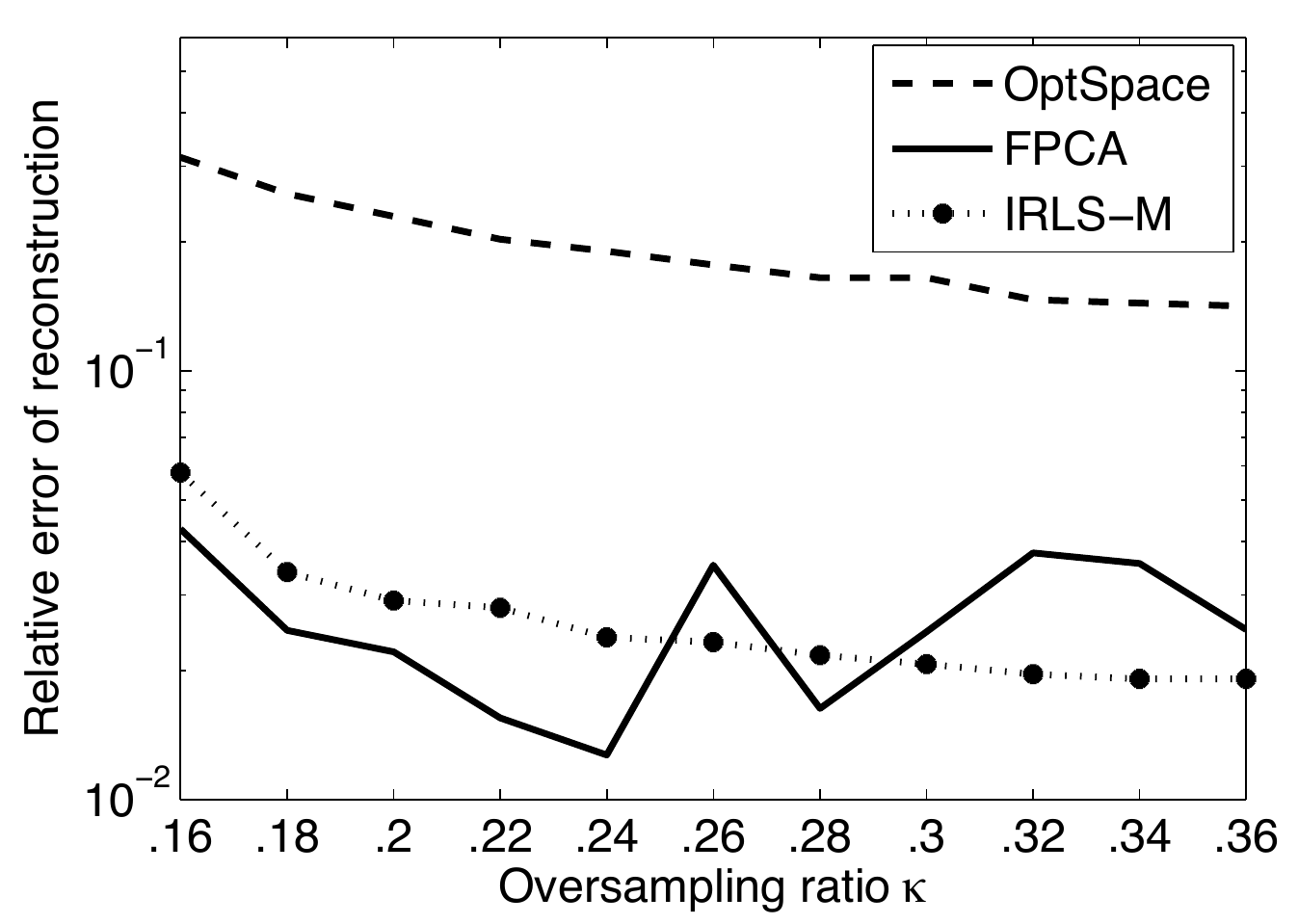} \label{4:a} 
}
\subfigure[Reconstruction time (s)]{ 
\includegraphics[scale =.52] {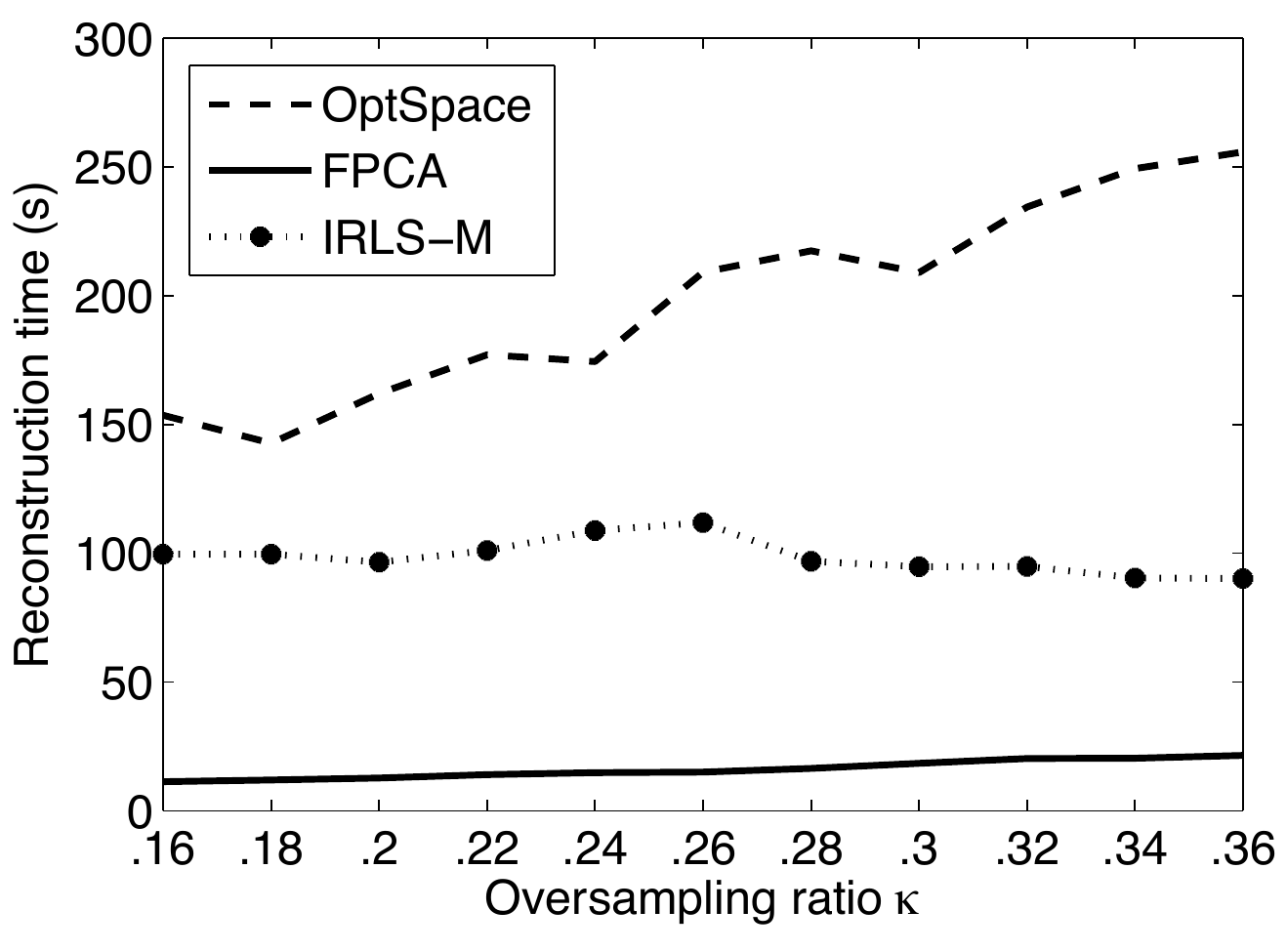} \label{4:b}
}

\caption{Comparison of IRLS-M, FPCA, and OptSpace over a class of randomly-generated rank-$30$ matrices subject to measurement noise.}
\medskip
\end{figure}


\paragraph{3}  In Figures \ref{fig:5}$.5$ and \ref{fig:6}$.6$ we use IRLS-M, FPCA, and OptSpace to reconstruct grayscale images from partial pixel 
measurements.  We are not suggesting that matrix completion algorithms should be used for image inpainting, but that we use images as a more realistic {\nnew{(nonrandom)}} set of matrices which are well-approximated by low-rank matrices.  In general, grayscale images correspond to matrices with sharply-decaying singular values.  In all experiments, the observed pixel measurements are chosen independently from the uniform distribution.   The displayed images were constructed by running each algorithm using the partial pixel measurements, and then thresholding all negative pixel values to be zero.

%
%

\begin{figure}[h]
\label{fig:5}

\medskip
\begin{center}
\subfigure[Original image]{ \includegraphics[scale=.3]{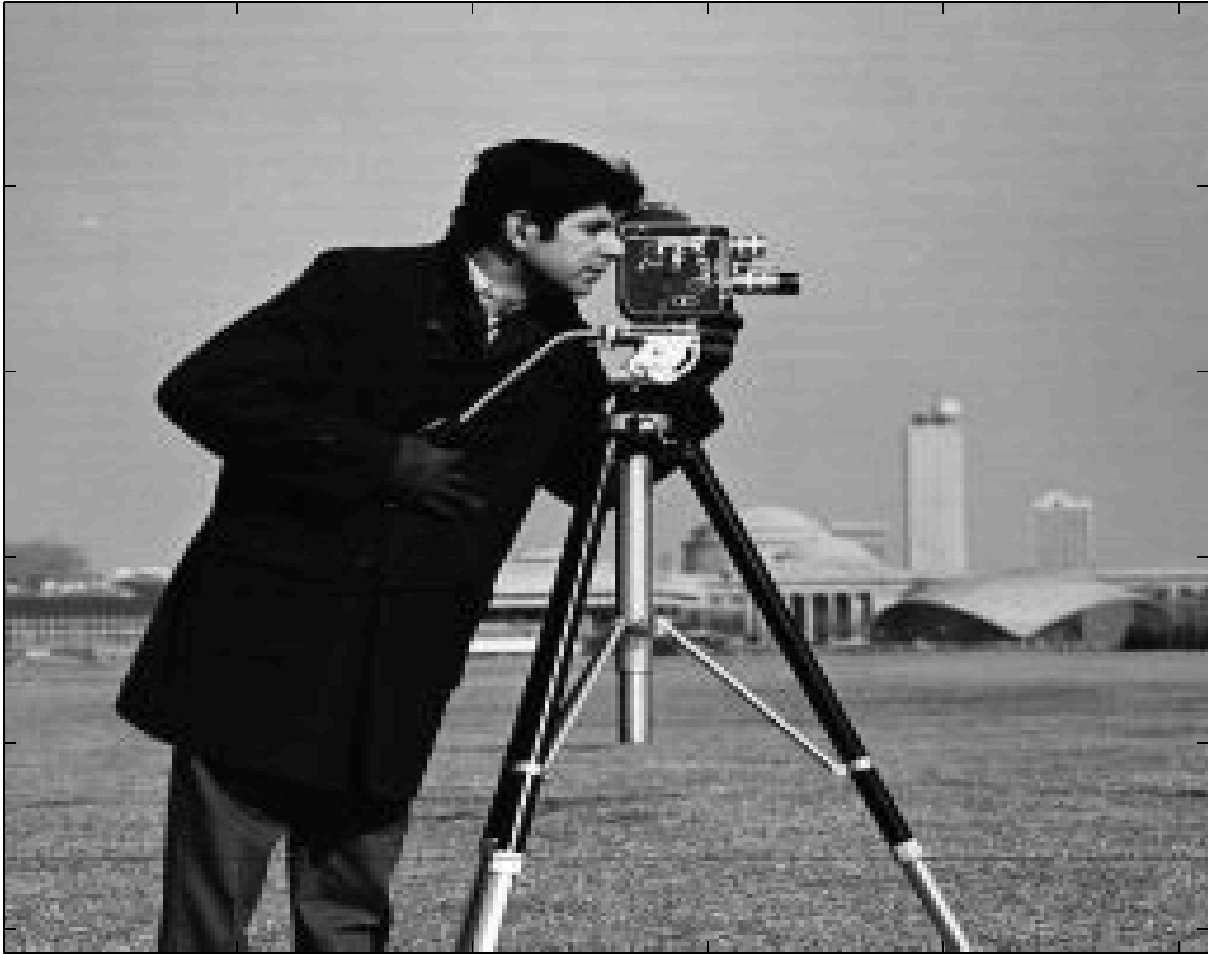} \label{5:a} }
\subfigure[Best rank-$16$ approximation]{ \includegraphics[scale =.3] {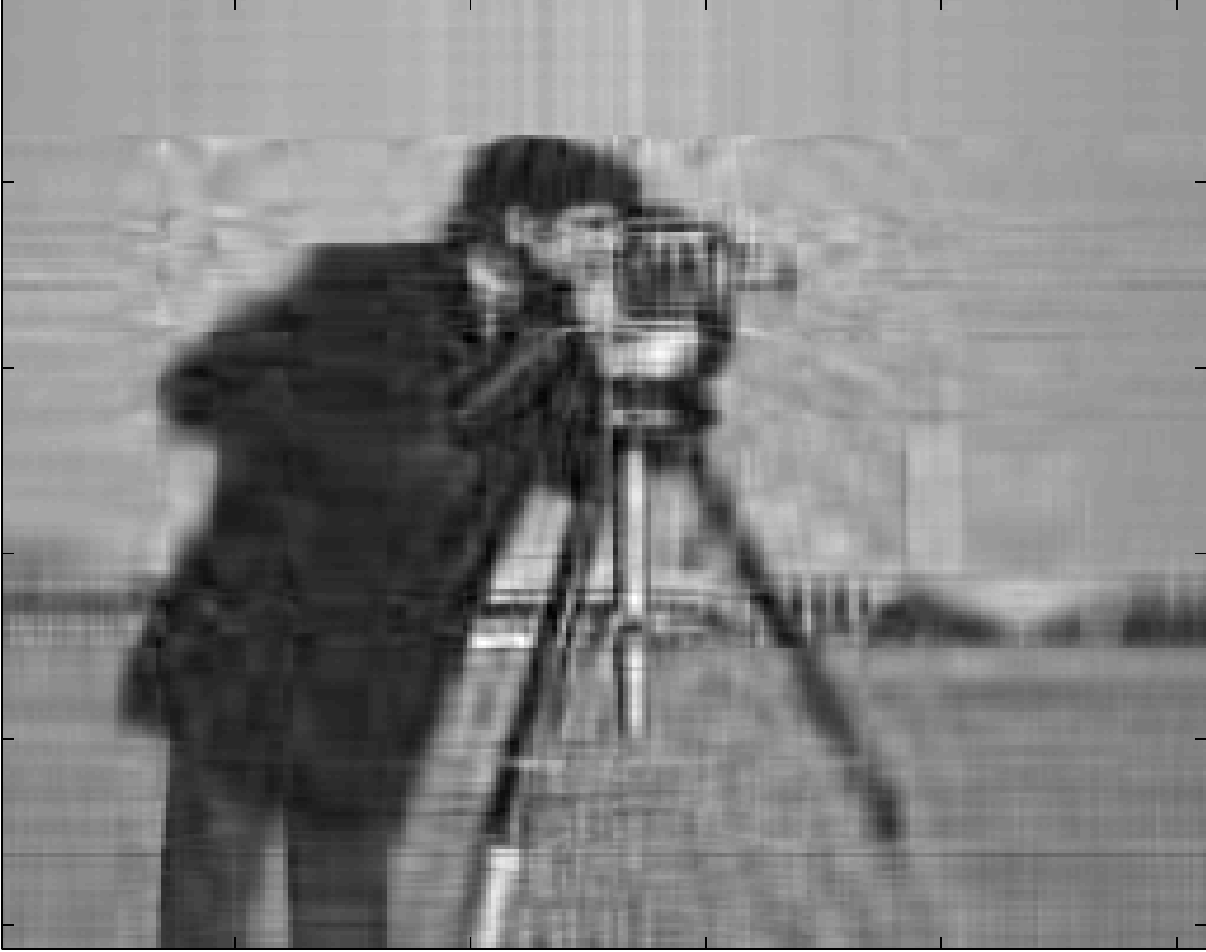} \label{5:b} }
\subfigure[$50 \%$ observed pixels]{ \includegraphics[scale =.3] {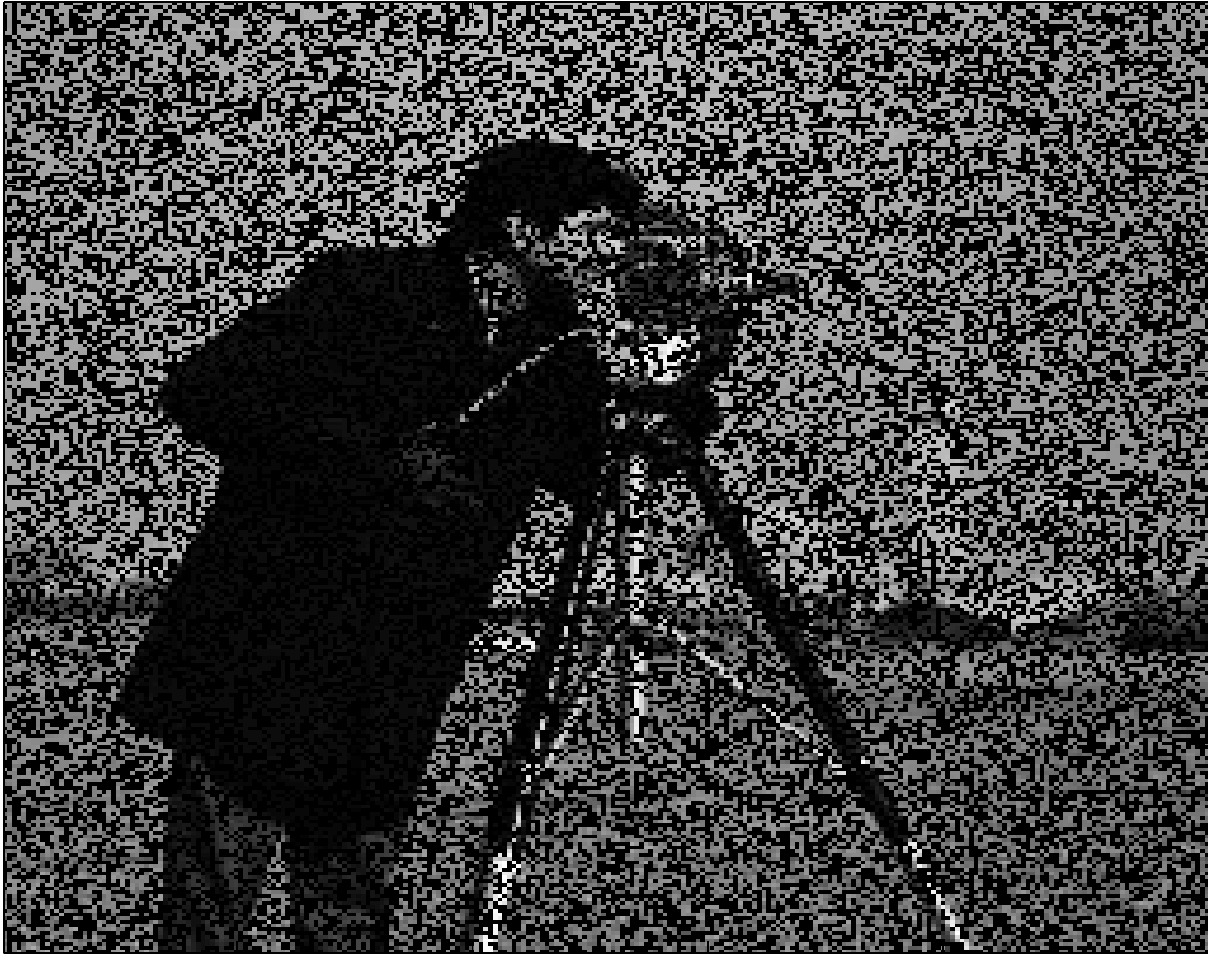} \label{5:c} }
\subfigure[IRLS-M reconstruction]{ \includegraphics[scale=.3]{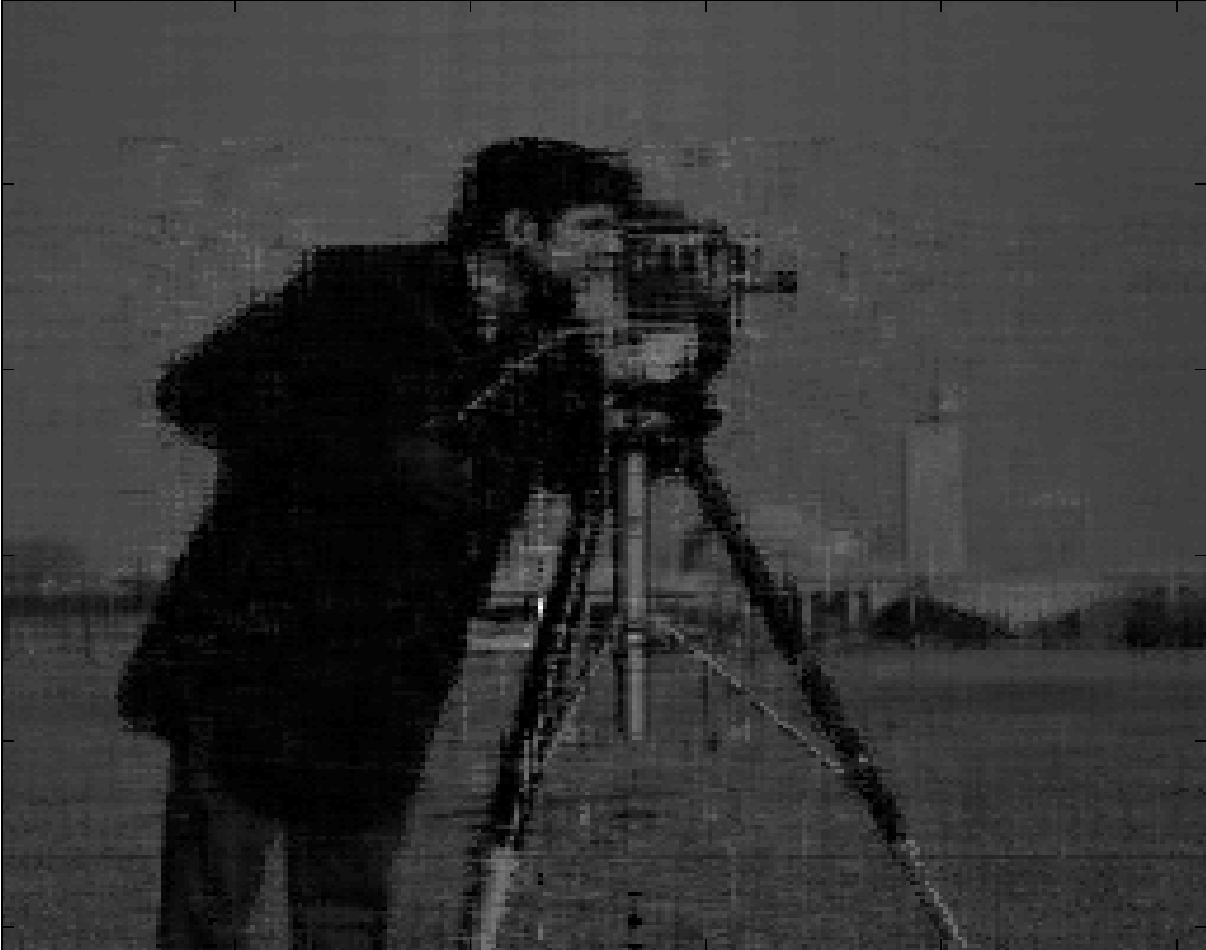} \label{5:d} }
\subfigure[FPCA reconstruction]{ \includegraphics[scale =.3] {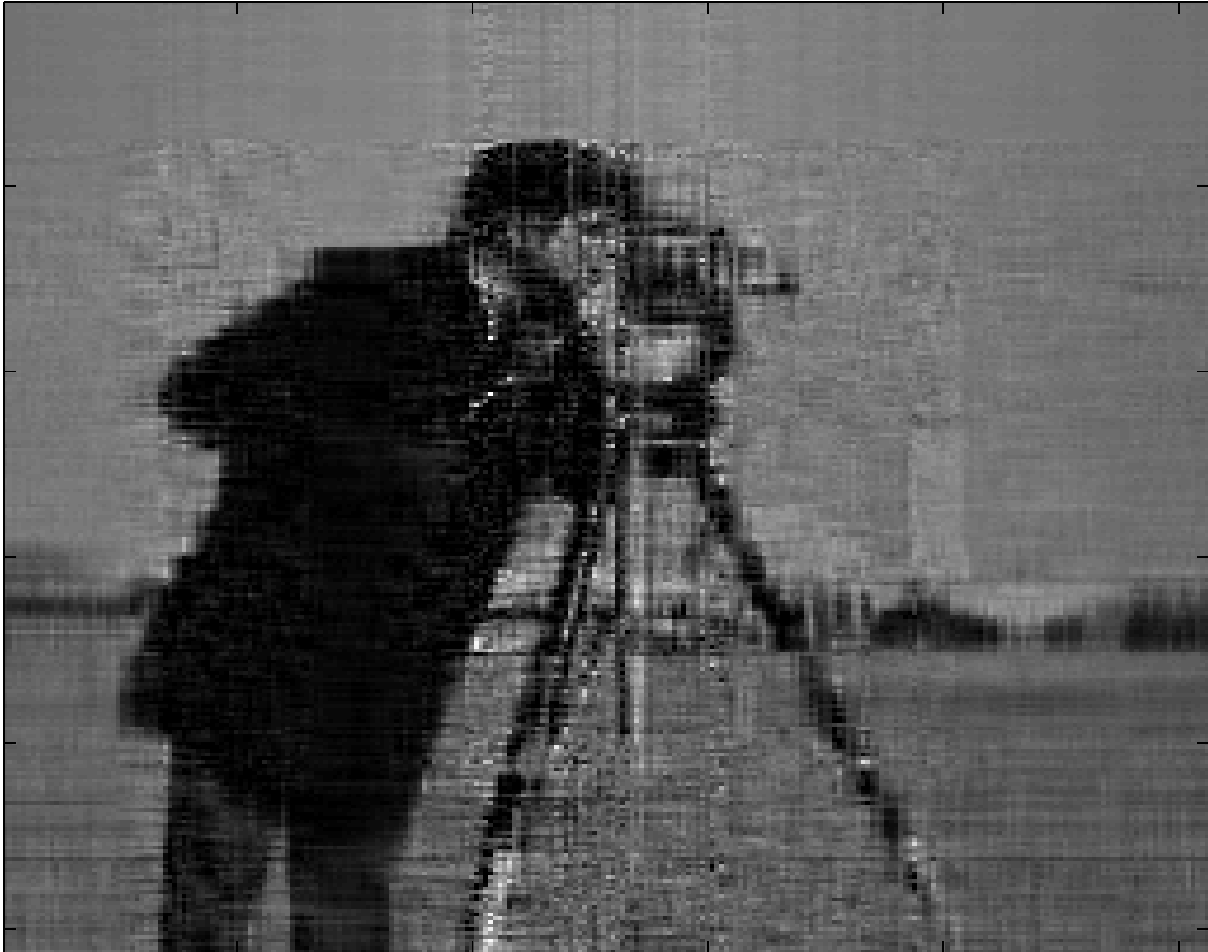} \label{5:e} }
\subfigure[OptSpace reconstruction]{ \includegraphics[scale =.3] {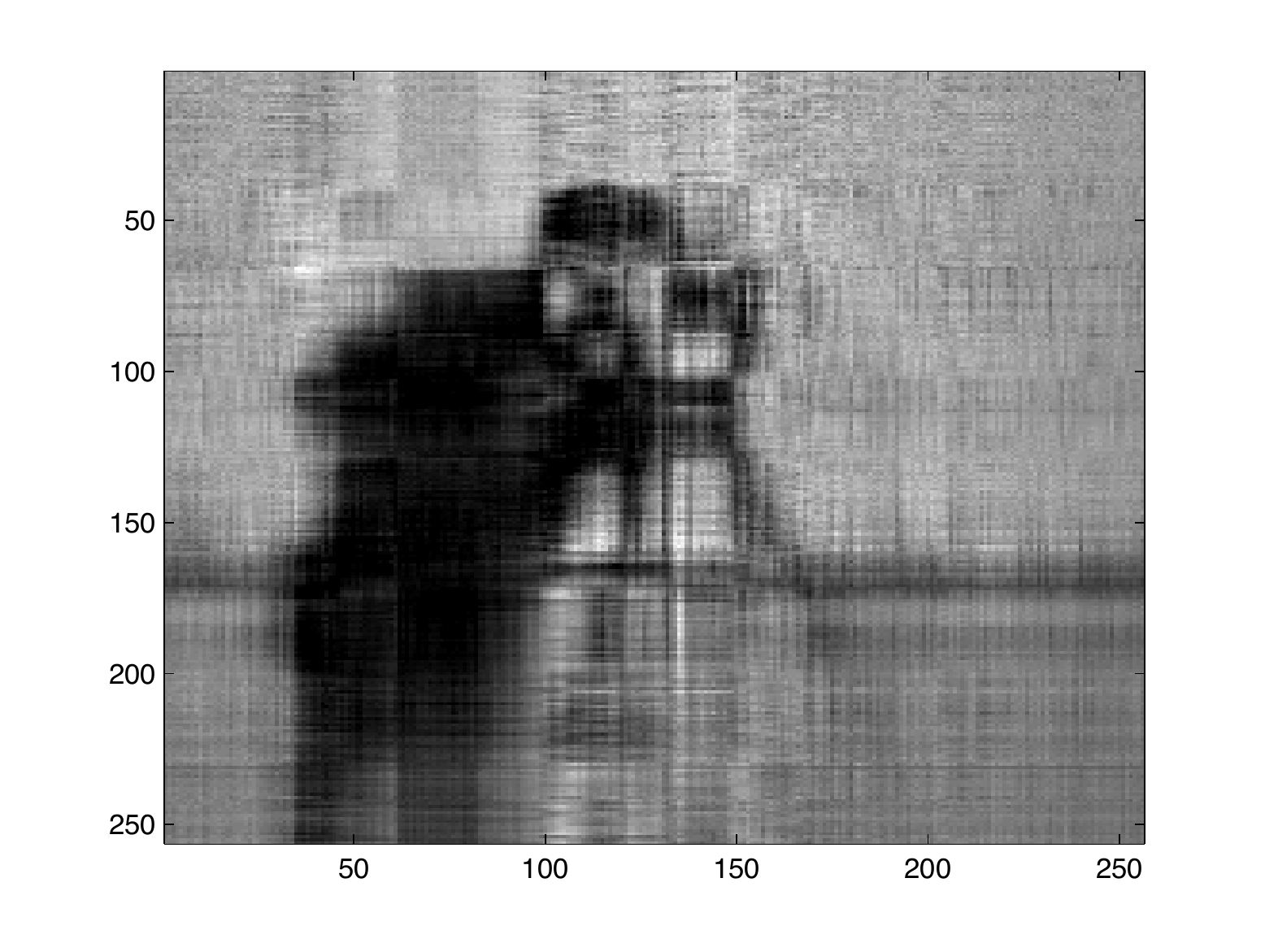} \label{5:f} }
 \\
 \caption{Figure \ref{5:a} is the original $256 \times 256$ Camera man image.  Figure \ref{5:b} is its best rank-$16$ approximation. Figure \ref{5:c} displays the $50\%$ randomly-distributed partial pixel measurements.  Figures \ref{5:d}, \ref{5:e}, and \ref{5:f} are the result of applying  IRLS-M, FPCA, and OptSpace, respectively, using these partial measurements and with input rank $r = 16$, followed by thresholding all negative pixel values to zero.  The relative error (with respect to the Frobenius norm) of the best rank-$16$ approximation, IRLS-M, FPCA, and OptSpace, respectively are $.1302$, $.1270$, $.1646$, and $.2157$.}
\end{center}

\end{figure}

\begin{figure}[h]
\label{fig:6}

\medskip
\begin{center}
\subfigure[Original image]{ \includegraphics[scale=.3]{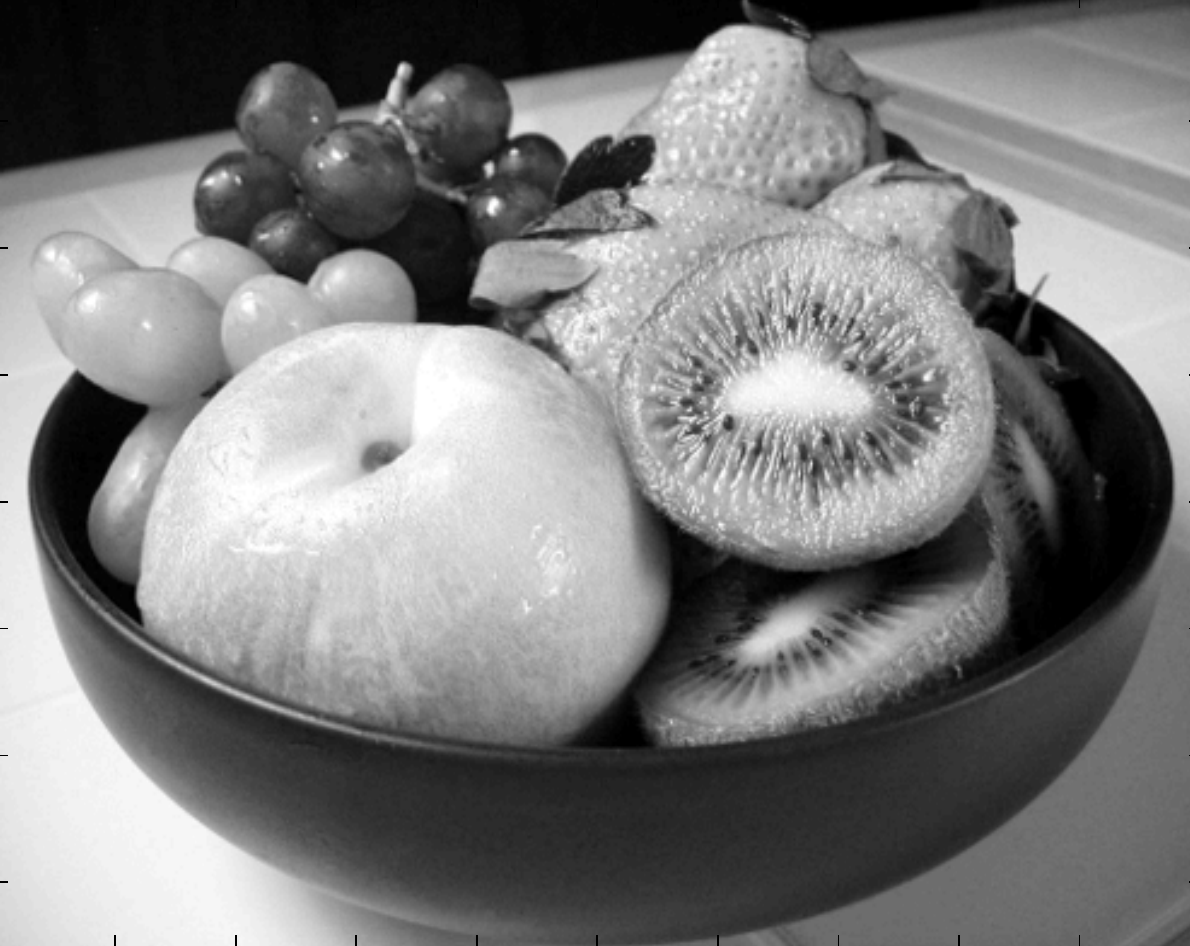} \label{6:a} }
\subfigure[Best rank-$20$ approximation]{ \includegraphics[scale =.3] {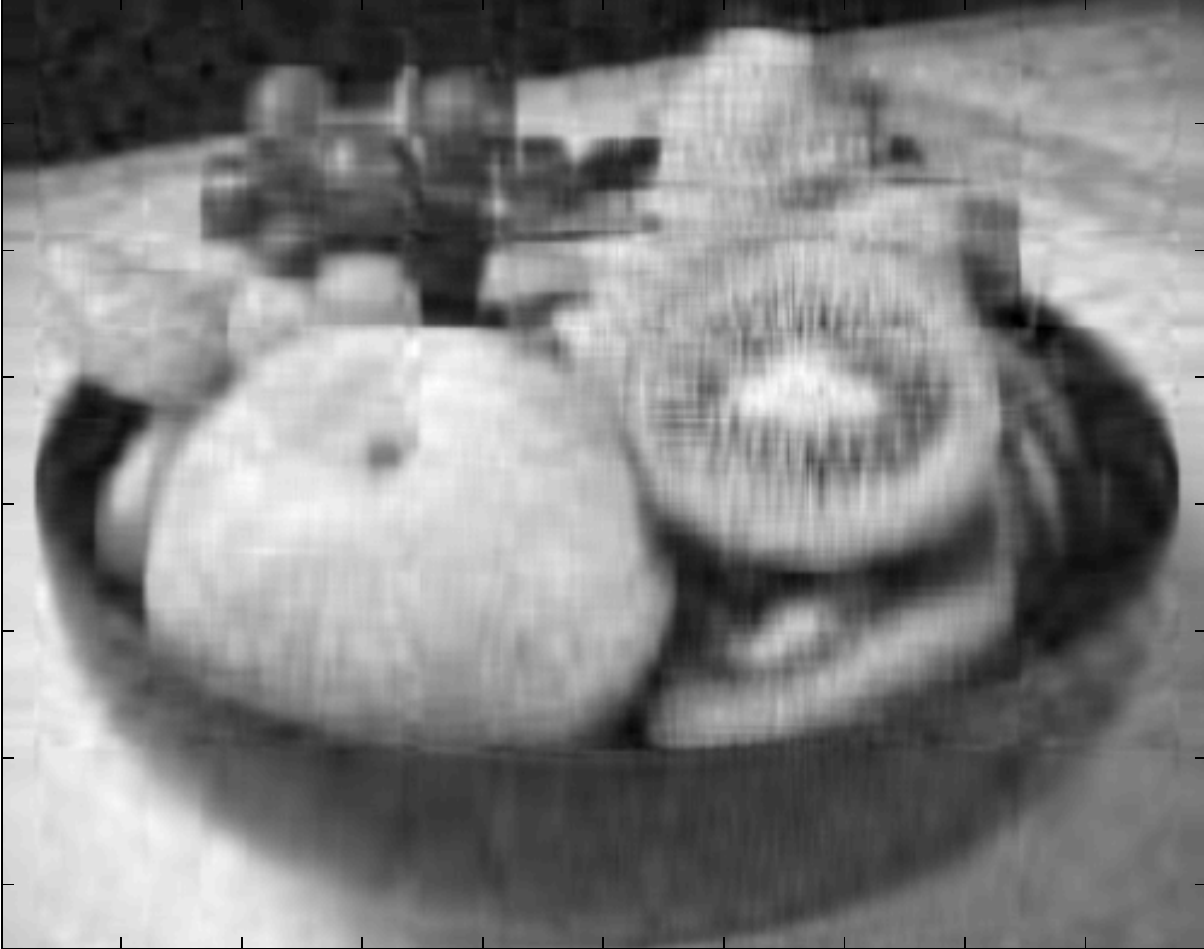} \label{6:b} }
\subfigure[$30 \%$ observed pixels]{ \includegraphics[scale =.3] {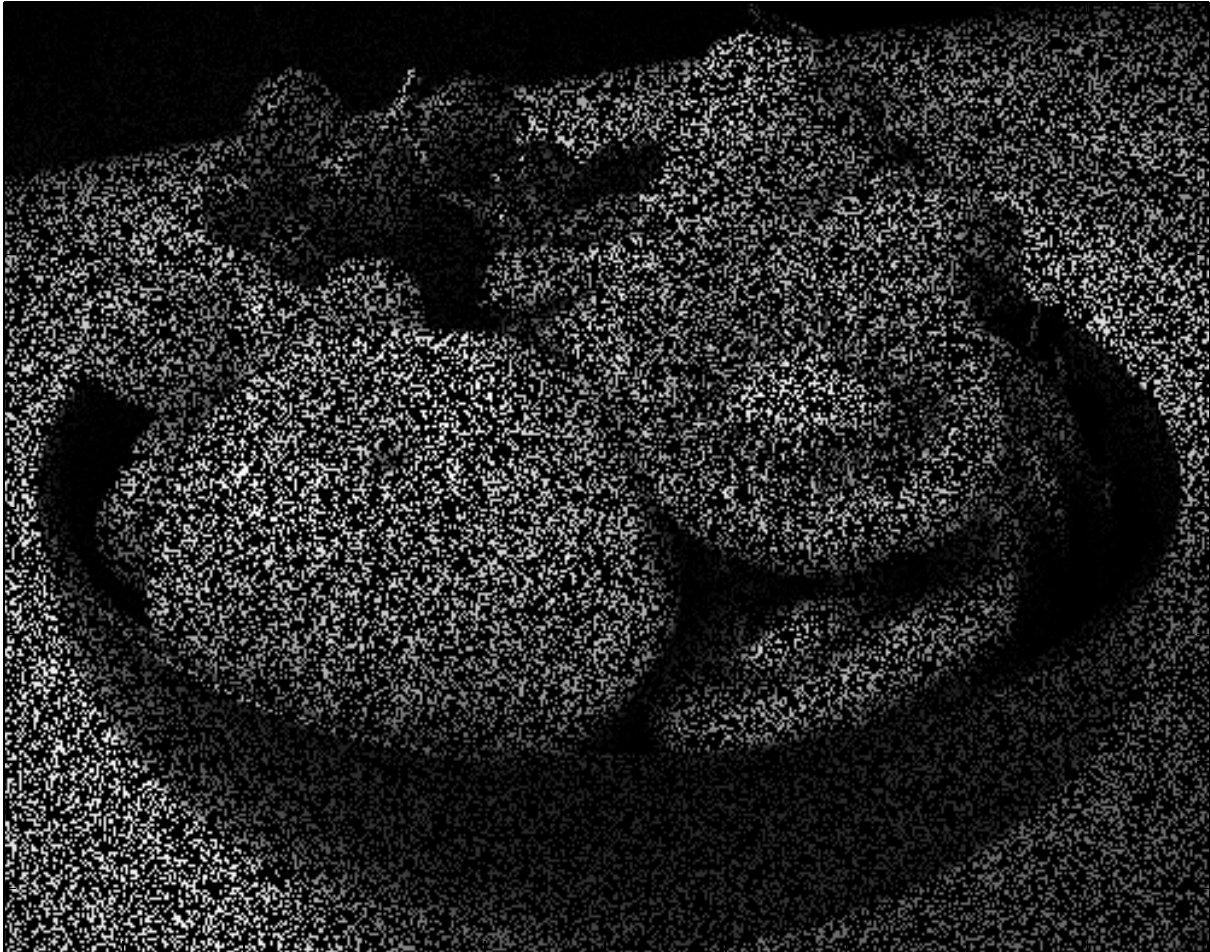} \label{6:c} }
\subfigure[IRLS-M reconstruction]{ \includegraphics[scale=.3]{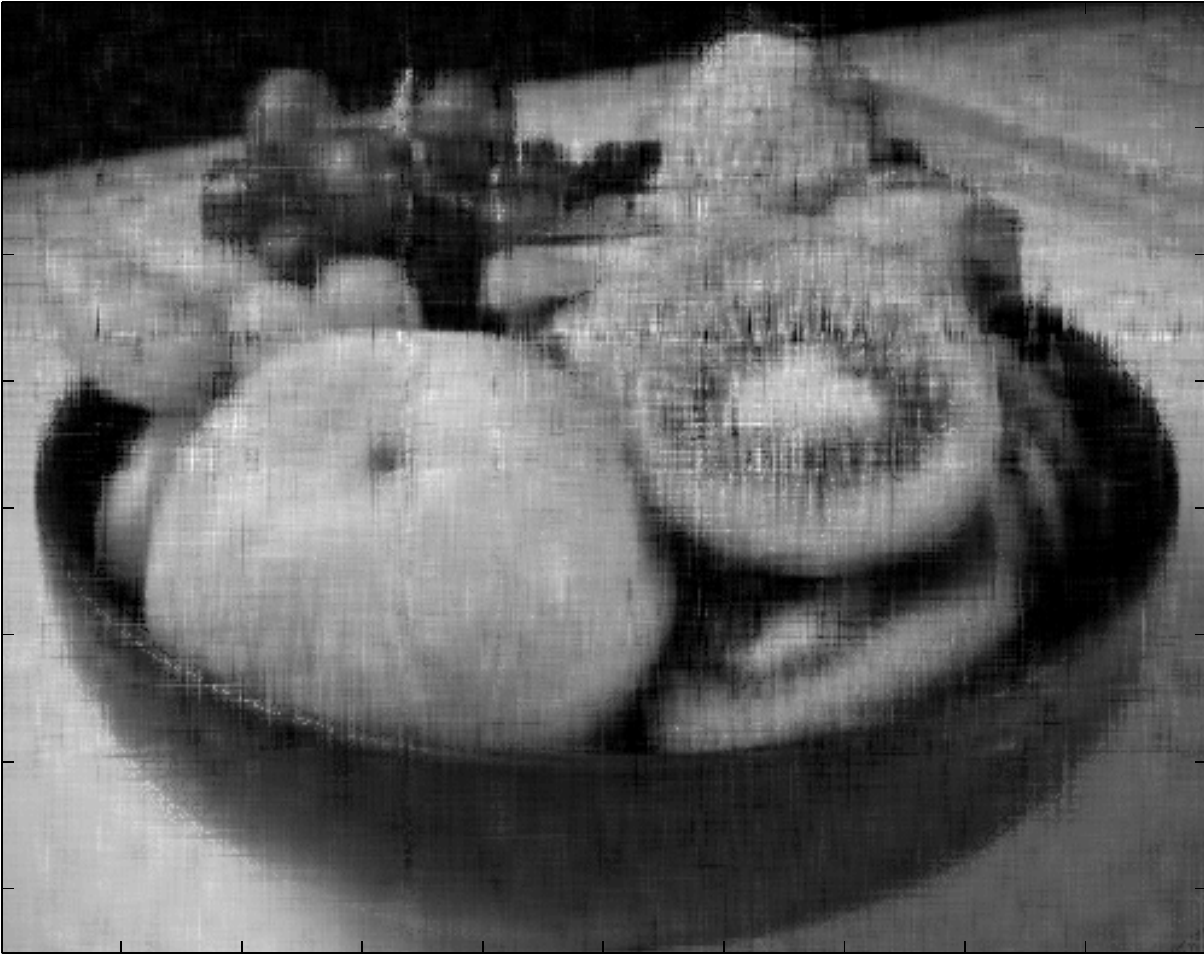} \label{6:d} }
\subfigure[FPCA reconstruction]{ \includegraphics[scale =.3] {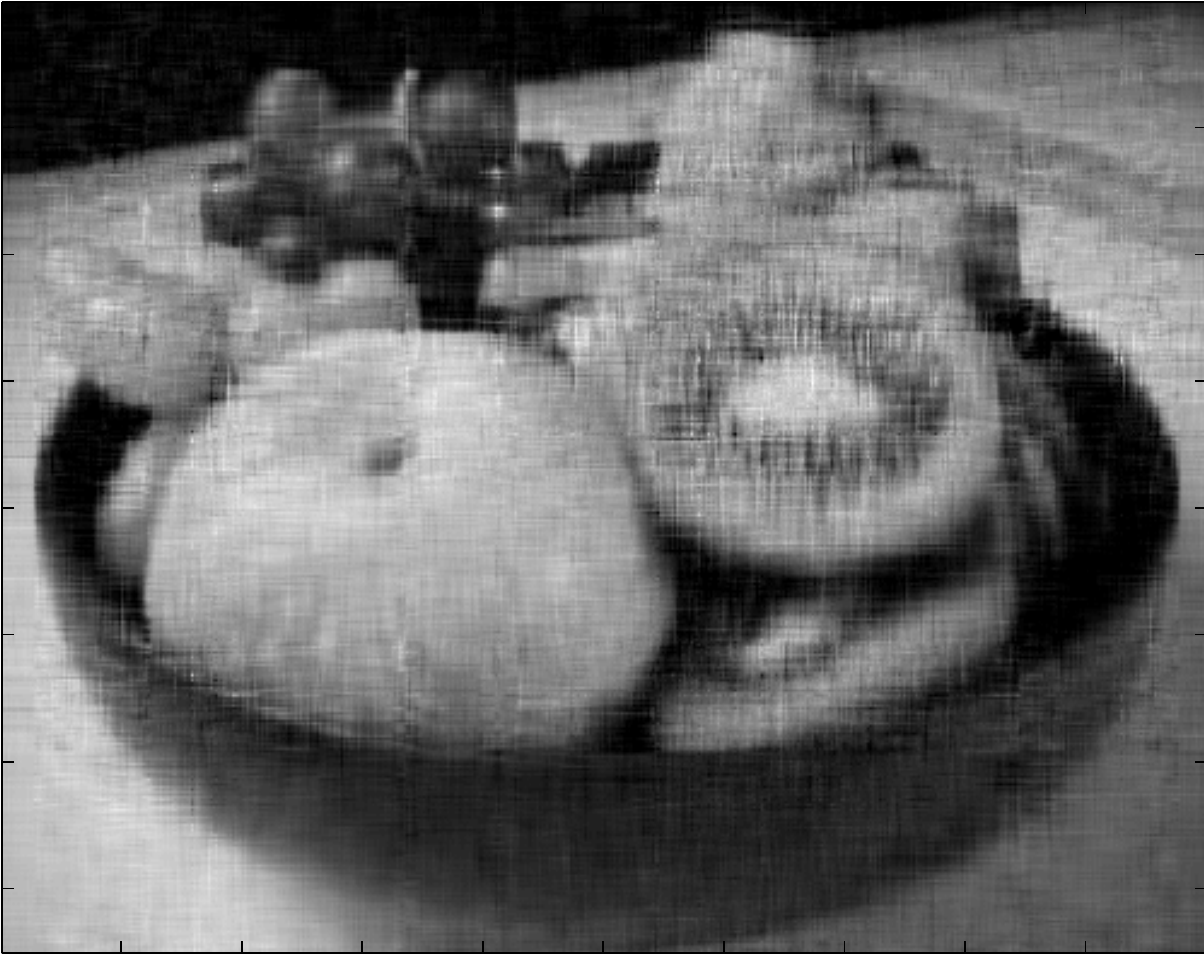} \label{6:e} }
\subfigure[OptSpace reconstruction]{ \includegraphics[scale =.3] {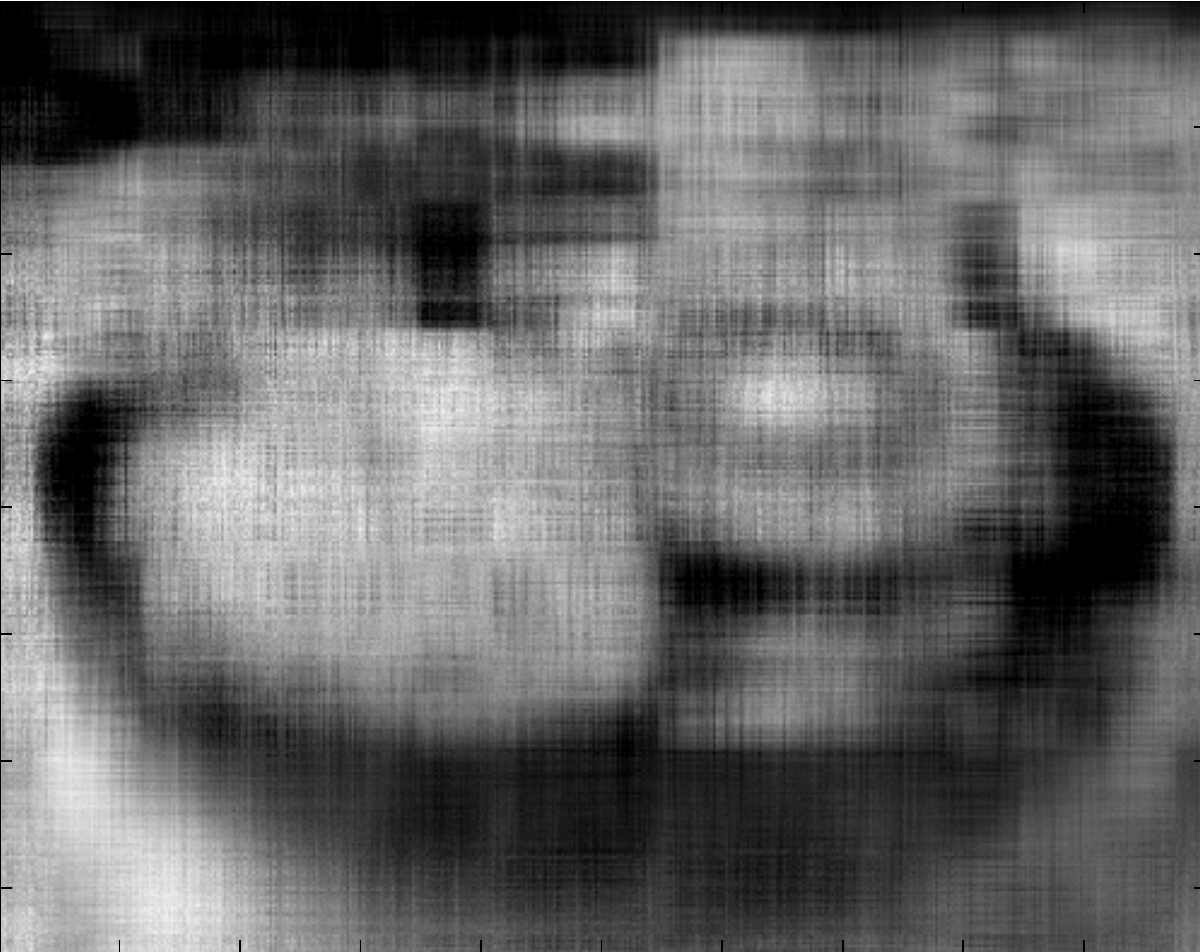} \label{6:f} }
 \\ 
 \caption{Figure \ref{6:a} is the original $375 \times 500$ Fruit Bowl image.  Figure \ref{6:b} is its best rank-$20$ approximation.  Figure \ref{6:c} displays the $30\%$ randomly-distributed partial pixel measurements.  Figures \ref{6:d}, \ref{6:e}, and \ref{6:f} are the result of applying IRLS-M, FPCA, and OptSpace, respectively, using these partial measurements with input rank $r = 20$, followed by thresholding all negative pixel values to zero. The relative error (with respect to the Frobenius norm) of the best rank-$20$ approximation, IRLS-M, FPCA, and OptSpace, respectively are $.0955$, $.1066$, $.1131$, and $.2073$.}

\end{center}
\end{figure}

\medskip

Our results indicate that IRLS-M outperforms OptSpace for reconstructing low-rank and approximately low-rank matrices whose singular values exhibit 
{\nnew{fast}} decay.  Our results indicate that IRLS-M is competitive with FPCA in accuracy and superior to FPCA in accuracy once the number of 
measurements is sufficiently high.  Nevertheless, IRLS-M can be slower in speed than FPCA.  Figures \ref{fig:5}$.5$ and \ref{fig:6}$.6$ suggest that IRLS-M
has the potential to reconstruct approximately low-rank matrices whose singular values exhibit sharp decay more accurately than FPCA.

\medskip

\begin{remark}
{\em
Much of the literature on matrix completion emphasizes the difference between ``hard" matrix completion problems  and ``easy" matrix completion problems, as determined by the parameters $n, p, k$, and $m$; generally speaking, ``hard" problems correspond to the situation where the number of measurements $m$ is close to the theoretical lower bound $m = C k \max\{n,p\} \log^2(\max\{n,p\})$.   For more details, we refer the reader to \cite{kemooh09}.  We note that the range of parameters $n=p, m,$ and $k$ considered in Figure \ref{fig:1}$.1$ and Figure \ref{fig:1}$.2$ corresponds to a transition from  ``hard"  to ``easy" problems. 
}
\end{remark}

\medskip

\begin{remark}
{\em
In all above numerical experiments, we ran the IRLS-M algorithm with parameter $\gamma=1$ as this greedy choice gives the best speed of convergence in practice, even though we could only prove convergence of the algorithm  for $\gamma=1/n$ as in Proposition \ref{prop:mini}.  We stopped the IRLS-M iterations when either $200$ iterations have been reached or $| \varepsilon_{\ell} - \varepsilon_{\ell-1} |/\varepsilon_{\ell} \leq 10^{-6}$ for more than $50$ consecutive iterations $\ell$.  These are the default parameters on the Matlab code we have made available online.\footnote{ \tt http://www.cims.nyu.edu/$\sim$rward/IRLSM.zip} {\nnew{(One part of the code is implemented in C because Matlab prohibited fast implementation of the sequence of problems in\eqref{woodbury3} below.)}}

We ran the OptSpace algorithm using the publicly-available {\nnew{Matlab}} code by the authors\footnote{\tt http://www.stanford.edu/$\sim$raghuram/optspace/code.html} and using the default parameters provided therein.
Finally, we ran the FPCA algorithm using publicly-available {\nnew{Matlab}} code by the authors\footnote{\tt http://www.columbia.edu/$\sim$sm2756/FPCA.htm}. We found that the running time of this algorithm improved considerably when we changed the settings to the ``easy" problem settings, even when the default setting was ``hard", and we report the performance of FPCA with this modification. 
}
\end{remark}
\medskip
\begin{remark}
{\em
All three algorithms are ``rank-aware", meaning that the underlying rank $r$ is provided as input.   In practice most matrices of interest are only approximately low-rank, and the best choice of $r$ is not known a priori.    In the image reconstruction examples in Figures \ref{fig:5}$.5$ and \ref{fig:6}$.6$, we picked a rank $k$ and number of measurements $m$ for each image from a collection of several different attempts based on visual inspection.  More generally, a systematic way to select the rank is using \emph{cross-validation}, {\nnew{see \cite{rw09} for details}}. 

 }
\end{remark}

\section{Computational Considerations}

\subsection{Updating W}  In order to perform the weight update, 
\[
\label{eq:weightup}
W^{\ell} = U^{\ell} ( \Sigma^{\ell}_{\varepsilon_\ell} )^{-1} (U^{\ell} )^*,
\]
one does not need to compute all $n$ singular vectors of $W^{\ell}$, but rather only the first $r$, where $r$ is the largest integer for which $\sigma_{r} > \varepsilon_{\ell}/\gamma$.   This is because $W^{\ell}$ decomposes as the sum of a rank $r$ matrix and a perturbation of the identity.  
In practice, it is sufficient to take $\gamma = 1$ (this was done in the numerical experiments of the last section); in this case, $r$ is on the order of $K$, and $W^{\ell}$ decomposes into the sum of a low-rank matrix and a perturbation of the identity.

\subsection{Updating X}
Consider now the matrix update, 
$$
\bar{X} = \argmin_{\A(X) =\M} \mathcal J(X,W).
$$
In the matrix completion setting, the operator $\A$ is {\it separable}, i.e., acts columnwise, 
$$
\A(X) = (\A_1 X_1,\dots, \A_p X_p) = (\M_1, \dots, \M_p) = \M \in \mathbb C^m,
$$
where $X_i$, $i=1,\dots,p$,  are the columns of $X$, and $\A_i$, $i=1,\dots, p$ are suitable matrices acting on the vectors $X_i$. In this case
$
\| W^{1/2} X \|_F^2 = \sum_{i=1}^p \| W^{1/2} X_i \|_{\ell_2^n}^2,
$
so that
$$
\bar{X} = \argmin_{\A(X) =\M} \mathcal J(X,W) \Leftrightarrow \bar{X}_i = \argmin_{\A_i X_i =\M_i} \| W^{1/2} X_i \|_{\ell_2^n}^2, \quad i=1,\dots,p.
$$
Note that
\begin{equation}
\label{col_update}
\bar{X}_i = W^{-1} \A_i^* (\A_i W^{-1} \A_i^*)^{-1} \M_i,  \quad i=1,\dots,p,
\end{equation}
where the actions here are simply matrix multiplications.  To speed up the computation of $(\A_i W^{-1} \A_i^*)^{-1} \M_j $, one may easily parallelize the independent column updates \eqref{col_update}. 

\subsubsection{The Woodbury matrix identity}\label{woodsec} The updates in \eqref{col_update} can be further expedited by exploiting the \emph{Woodbury matrix identity}, first introduced in
\cite{wo50}. It states that for dimensions $r \leq n$, if $A \in M_{n \times n}$, $U \in M_{n \times r}$, $C \in M_{r \times r}$, and $V \in M_{n \times r}$, then
\begin{equation}
\label{WMI}
(A + U C V^*)^{-1} = A^{-1} - A^{-1} U(C^{-1} + V^* A^{-1} U)^{-1} V^* A^{-1},
\end{equation}
assuming all of the inverses exist.  

We may write the inverse weight matrix $W^{-1}$ as the sum of a diagonal rank-$r$ matrix and a perturbation of the identity:
$$
W^{-1} = \varepsilon I + U \widetilde{\Sigma}_{[r]} U^*
$$
where $\widetilde{\Sigma}_{[r]} =(\Sigma - \varepsilon I)_{[r]} =$ diag$(\max\{\sigma_j - \varepsilon, 0 \})$.  
If $\gamma = 1$ is set in the algorithm, then $W^{-1}$ decomposes into the sum of a $K$-rank matrix and a perturbation of the identity during iterations $\ell$ where $\varepsilon_{\ell} < \varepsilon_{\ell-1}$. The matrix $\A_i W^{-1} \A_i^*$ may then be expanded as
\begin{equation}
\label{eq:invert}
\A_i W^{-1} \A_i^* = \varepsilon \A_i \A_i^* +   \big( \A_i U \big) \widetilde{\Sigma}_{[r]} (\A_i U)^*.
\end{equation}
Recall that $\A_i \in M_{m_i \times n}$, and that $\sum_{j=1}^{p} m_i = m$.  
Setting $Q_i = \A_i \A_i^*$ and $R_i = \A_i U$, and applying the Woodbury matrix identity,
\begin{equation}
\label{wood}
( \A_i W^{-1} \A_i^* )^{-1} = \varepsilon^{-1} Q_i^{-1}\Big[I - R_i \Big( \varepsilon \widetilde{\Sigma}_{[r]}^{-1} + R_i^* Q_i^{-1} R_i \Big)^{-1} R_i^* Q_i^{-1} \Big].
\end{equation} 
As $\widetilde{\Sigma}_{[r]}$ is a diagonal matrix, its inverse is trivially computed.   In the setting of matrix completion, the operator $\A_i$ is simply a subset of rows from the $n \times n$ identity matrix, and $\A_i \A_i^{T} = I \in M_{m_i \times m_i}$, and $U_i=R_i=\A_i U\in M_{m_i \times r}$ is a subset of $m_i$ rows from the matrix $U$. Therefore, utilizing the formula \eqref{wood}, this reduces to 
\begin{eqnarray}
\label{woodbury3}
( \A_i W^{-1} \A_i^* )^{-1} &=& \varepsilon^{-1} \Big[I - U_i \big( \varepsilon \widetilde{\Sigma}_{[r]}^{-1} + U_i^* U_i\big)^{-1} U_i^*\Big],
\end{eqnarray}
and one essentially needs to solve a linear system corresponding to the matrix $\widetilde{\varepsilon \Sigma}_{[r]}^{-1} + U_i^* U_i \in M_{r \times r}$ as opposed to a system involving the inversion of a matrix in $M_{m_i \times m_i}$.  When $\gamma = 1$ and $r = K+1$, such as at iterations $\ell$ where $\varepsilon_{\ell} < \varepsilon_{\ell-1}$, the average number of measurement entries in any particular column is at least $\widehat{m} = O(k \log{np}) \gg k$, so this amounts to a great reduction in computational cost.

\section{Variational Interpretation}

\subsection{A functional of matrices and its derivatives}

In this section we show that the low-rank matrix recovery can be reformulated as an alternating minimization of a functional of matrices.   Assume $X \in M_{n \times p}$, $0\prec W=W^* \in M_{n \times n}$, and consider the functional,
\begin{equation}
\label{func}
\J(X,W) := \frac{1}{2} \left ( \| W^{1/2} X  \|_F^2 + \| W^{-1/2}\|_F^2 \right ).
\end{equation}
In the following subsections we shall prove that the IRLS-M algorithm may be reformulated as follows.

\medskip

\noindent 
\fbox{
\begin{minipage}{14cm}
{\rm {\bf IRLS-M algorithm for low-rank matrix recovery: } Initialize by taking $W^0:=I \in M_{n \times n}$.  Set $\varepsilon_0:=1$,  $K\geq k$, and $\gamma >0$. Then recursively define, for $\ell = 1,2,\dots,$
{\nnew{\begin{align}
X^{\ell}&:=\arg \min_{\A(X)=\M } \mathcal J (X, W^{\ell}), \notag\\
\label{en3}
\varepsilon_{\ell} & :=\min \left\{\varepsilon_{\ell-1}, \gamma \s_{K+1}(X^{\ell}) \right\}.
\end{align}}}
The up-date of the weight matrix $W^\ell$ follows a variational principle as well, i.e.,
\begin{eqnarray}
\label{wn2}
W^{\ell} &:=& \arg \min_{0 \prec W = W^* \preceq \varepsilon_{\ell}^{-1} I} {\cal J}(X^{\ell},W) 
\end{eqnarray}
The algorithm stops if $\varepsilon_\ell=0$; in this case, define $X^j:=X^\ell$ for 
$j>\ell$.  In general, the algorithm generates an infinite sequence 
$(X^\ell)_{\ell\in \mathbb N}$ of matrices. }
\end{minipage}
}

\subsubsection{Optimization of $\cal J$ with respect to $X$}

Let us address the matrix optimization problems and explicitly compute their solutions. For consistency of notation, we need to introduce the following {\it left-multiplier operator} $\mathcal W^{-1} : M_{n \times p} \to M_{n \times p}$ defined by $\mathcal W^{-1}(X) := W^{-1} X$. With this notation we can write explicitly the solution of the minimization of $\mathcal J$ with respect to $X$.

\begin{lemma}
\label{lm2}
Assume that $W \in M_{n \times n}$ and $W=W^* \succ 0$.
Then the  minimizer
$$
\bar{X} = \argmin_{\A(X) =\M} \mathcal J(X,W)
$$
is given by
$$
\bar{X} = \left \{ \mathcal W^{-1} \circ \A^* \circ [ \A \circ \mathcal W^{-1} \circ \A^*]^{-1} \right \}(\M).
$$
\end{lemma}

\begin{proof}
For $W=W^* \succ 0$, we have
$$
\bar{X} =\argmin_{\A(X) = \M} \| W^{1/2} X \|_F^2
$$
if and only if there exists a $\lambda \in \mathbb{C}^m$ such that
\begin{equation}
\label{optcond}
\A^*(\lambda) = W \bar{X}, \quad \textrm{or } \langle W \bar{X} , H \rangle =0, \mbox{ for all } H \in \ker \A.
\end{equation}
This equivalence can be shown directly: assume $\bar{X}$ optimal and $H \in \ker \A$, then
\begin{eqnarray*}
\| W^{1/2} \bar{X} \|_F^2 &=&\tr(W \bar{X} \bar{X}^*) \leq  \tr(W (\bar{X} +H) (\bar{X}+H)^*)\\
&=& \tr(W \bar{X} \bar{X}^*) + \tr(W H H^*) + \tr[W ( H \bar{X}^* + \bar{X}H^*)]. 
\end{eqnarray*}
Actually $\bar{X}$ is optimal if and only if 
$\tr[W ( H \bar{X}^* + \bar{X}H^*)] =0$ for all $H \in \ker \A$. 
By the properties of the trace we have
$$
\tr(W H \bar{X}^* ) = \overline{\tr[(W H \bar{X}^*)^*]} =  \overline{\tr(\bar{X} H^* W)} =  \overline{\tr(W \bar{X} H^*)}=  \overline{\tr(H^* W \bar{X})}.
$$
Hence $\tr[W ( H \bar{X}^* + \bar{X}H^*)]=0$ if and only if $\Re e (\langle W \bar{X} , H \rangle) =0$. 
This observation allows us to compute $\bar{X}$ explicitly.
If we assume for a moment that actually $\langle W \bar{X} , H \rangle =0$ for all $H \in \ker \A$, then necessarily $ W \bar{X} = \A^*(\lambda)$, and we obtain $\bar{X} = \mathcal W^{-1} \circ \A^*(\lambda)$. 
Hence $\M = \A(\bar{X}) = \A \ \circ \mathcal W^{-1} \circ \A^*(\lambda)$ and $\lambda = [ \A \circ \mathcal W^{-1} \circ \A^*]^{-1} (\M)$. 
Eventually $\bar{X} = \left \{ \mathcal W^{-1} \circ \A^* \circ [ \A \circ \mathcal W^{-1} \circ \A^*]^{-1} \right \} (\M)$ satisfies by construction $\langle W \bar{X} , H \rangle =0$, hence $\Re e (\langle W \bar{X} , H \rangle) =0$, and the optimality condition holds.
\end{proof}


\subsubsection{Optimization of $\mathcal J$ with respect to $W$}

In this section we address the solution of the following optimization problem
$$
\bar W = \arg \min_{0 \prec W = W^* \preceq \varepsilon^{-1} I} {\cal J}(X,W) .
$$
Let us collect a few preliminary lemmas in this direction.
{\bnew \begin{lemma}
\label{lm3}
Under the constraint $W =W^*  \succ 0$, the Fr\'echet derivative of $\mathcal J$ with respect to its second variable is
\begin{equation}
\label{deriv}
\partial_W \mathcal J(X,W) = X X^* - W^{-2}.
\end{equation}
\end{lemma}
\begin{proof}
Let us consider the G\^{a}teaux differentiation with respect to {\nnew{$V = V^*$}}, 
\begin{eqnarray*}
\frac{ \partial}{\partial V} \| W^{1/2} X\|_F^2 &:=& \lim_{h \to 0} \tr \left [ \left ( \frac{(W + h V) - W}{h}\right ) X X^* \right] \\
&=& \tr (V X X^*) = \langle X X^*, V \rangle.
\end{eqnarray*}
Since the self-adjoint and positive-definite matrices {\nnew{form}} an open set, we have
\begin{equation}
\label{XX}
\partial_W \| W^{1/2} X\|_F^2 = X X^*.
\end{equation}
It remains to address the  Fr\'echet derivative of $\| W^{-1/2}\|_F^2 = \tr (W^{-1})$. We have
\begin{eqnarray*}
\frac{\partial}{\partial V} \| W^{-1/2}\|_F^2 &=& \lim_{h \to 0} h^{-1} \left \{ \tr [ (W + h V)^{-1}] - \tr(W^{-1}) \right \} \\
&=& \tr \left [ \lim_{h \to 0} \frac{   (W + h V)^{-1} - W^{-1}}{h}\right ].
\end{eqnarray*}
Note that
$
(W + h V)\left [ \frac{ (W + h V)^{-1} - W^{-1}}{h} \right ] = - V W^{-1},
$
so
$$
\lim_{h \to 0} \frac{   (W + h V)^{-1} - W^{-1}}{h} = -W^{-1} V W^{-1}. 
$$
By using the cyclicity of the trace we obtain
$
\partial_W \| W^{-1/2}\|_F^2 = \langle W^{-2}, V \rangle,
$
and therefore
\begin{equation}
\label{W2}
\frac{ \partial}{\partial W} \| W^{-1/2}\|_F^2 = -W^{-2}.
\end{equation}
This latter statement can also be obtained by applying
the more general result \cite[Proposition 6.2]{lese05}.
Putting together \eqref{XX} and \eqref{W2} yields \eqref{deriv}.
\end{proof}
}

Using a similar argument as in Lemma \ref{lm3}, we arrive at the next result. 
\begin{proposition}\label{lm3'}
Consider $X = U \Sigma V^*$, where $\Sigma = \diag(\sigma_1, ..., \sigma_n)$.  Let $\varepsilon > 0$.  Then
\begin{equation}\label{minJ}
\argmin_{0 \prec W=W^* \preceq \varepsilon^{-1} I} \mathcal J(X,W) =: \bar{W} = U \Sigma_{\varepsilon}^{-1} U^*.
\end{equation}
{\bnew We recall that $\Sigma_{\varepsilon} = \diag(\max\{\sigma_j, \varepsilon\})$.}
\end{proposition}

\begin{proof} First observe that $\mathcal J(X,W) = \infty$ if $W$ is not invertible so that 
\[
\bar{W} = \argmin_{0 \preceq W=W^* \preceq \varepsilon^{-1} I} \mathcal J(X,W). 
\]
That is, the constraint of positive-definiteness can be relaxed to positive semi-definiteness.
Now, the set 
$
{\cal C}_0 = \{ W \in M_{n \times n}: 0 \preceq W = W^* \preceq \varepsilon^{-1} I  \} 
$
is a convex set, and may be reformulated as 
\begin{eqnarray}
{\cal C}_0 &=& \{ W \in M_{n \times n}: 0 \preceq W = W^*, \hspace{2mm} \| W \|_{op} \leq \varepsilon^{-1}\} \nonumber \\
&=&  \{ W \in M_{n \times n}: 0 \preceq W = W^*, \hspace{2mm} \sup_{\| u \|_{\ell_2^n} = 1} \|  W^{1/2} u \|^2_{\ell_2^n} = \sup_{\| u \|_{\ell_2^n} = 1} \|  W u \|_{\ell_2^n} \leq \varepsilon^{-1} \}. \nonumber
\end{eqnarray}
Let $e_i \in \mathbb{R}^{n}$ denote the $i^{th}$ canonical basis vector, and note that $\| U e_i \|_{\ell^n_2} = 1$.  Then it is clear that ${\cal C}_0 \subset {\cal C}_1$, where ${\cal C}_1$ is the convex set
$$
 {\cal C}_1 =  \{ W \in M_{n \times n}: 0 \preceq W = W^*, \hspace{2mm} \| W^{1/2} U e_i \|_{\ell_2^n}^2 \leq   \varepsilon^{-1},  \hspace{5mm} i = 1,..., n \}.
$$  
The Lagrangian corresponding to the optimization of $\mathcal J$ over the convex set  ${\cal C}_1$ is given by
$$
\mathcal L(X,W, \lambda) := \frac{1}{2}\left[ \| W^{1/2} X \|_F^2 + \| W^{-1/2} \|_F^2 + \sum_{i=1}^n \lambda_i \big( \| W^{1/2} U e_i \|_{\ell_2^n}^2 - \varepsilon^{-1} \big) \right];
$$
by weak duality \cite[p.\ 226]{bova04} the minimum of ${\cal J}(X,\cdot)$ over ${\cal C}_1$ can be bounded from below by the Lagrangian dual function,
\begin{equation}
\nonumber
\sup_{(\lambda_1, ..., \lambda_n) \geq 0} g(\lambda) \leq \min_{W \in {\cal C}_1} {\cal J}(X,W),  \hspace{10mm} g(\lambda) := {\min_{W \in {\cal C}_1}} \mathcal{L}(X,W,\lambda).
\end{equation}
Following an argument similar to Lemma \ref{lm3}, $\bar{W}_{\lambda} = \arg \min_{W \in {\cal C}_1} \mathcal L(X,W,\lambda)$ agrees with the solution to the Karush-Kuhn-Tucker conditions,
\begin{equation}
\label{EL}
X X^* + \sum_{i=1}^n \lambda_i U E_{ii} U^* = W^{-2},
\end{equation}
or
$$
W = U \diag{(\sigma_i^2 + \lambda_i)^{-1/2}} U^*,
$$
if $W \in {\cal C}_1$, or equivalently if $\lambda_i \geq \max\{ \varepsilon^2 - \sigma^2_i, 0\}$. A straightforward calculation shows that $\bar{\lambda} = \arg \max_{\lambda_i \geq \max\{ \varepsilon^2 - \sigma^2_i, 0\}} g(\lambda)$ satisfies
$$
\bar{\lambda}_i =\left\{  \begin{array}{cc} \varepsilon^2 - \sigma_i^2, & \varepsilon^2 - \sigma_i^2 \geq 0, \nonumber \\
0, & \textrm{else},
 \end{array} \right.
$$
while $g(\bar{\lambda}) =\frac{1}{2} \sum_{\sigma_i < \varepsilon} \big(\frac{\varepsilon^2 + \sigma_i^2}{\varepsilon} \big) +\sum_{\sigma_i \geq \varepsilon} \sigma_i$.  But at $\bar{W}_{\bar \lambda} = U \diag{(\sigma_i^2 + \bar{\lambda}_i)^{-1/2} }U^*$, we have the equality $g(\bar{\lambda}) = {\cal J}(\bar{W}_{\bar \lambda}, X)$.  By duality, $\bar{W}_{\bar \lambda}$ must realize the minimal value of $\mathcal J(X, \cdot)$ over the set ${\cal C}_1$.
As $\bar{W}_{\bar \lambda}$ is also contained in ${\cal C}_0$ and as ${\cal C}_0 \subset {\cal{C}}_1$, it follows that $\bar{W}_{\bar \lambda}$ also minimizes ${\cal J}(X, \cdot)$ over the original convex set ${\cal C}_0$. We set $\bar W = \bar W_{\bar \lambda}$.
\end{proof}

\section{Analysis of the Algorithm}

\subsection{Basic properties of the iterates}

We start with simple properties of the iterates that hold without {\nnew{any}} requirements. In order to show the actual convergence of the {\nnew{IRLS-M}} algorithm
along with the relationship to nuclear norm minimization, we will assume additional properties of the measurement map $\A$ in the
next subsection.

\begin{proposition}
\label{prop2}
Assume that $(X^\ell,W^\ell,\varepsilon_\ell)_{\ell \in \mathbb N}$ is the output of the main algorithm. Then the following properties hold.
\begin{itemize}
\item[(i)] $\mathcal J(X^{\ell+1},W^{\ell+1}) \leq \mathcal J(X^\ell,W^\ell)$, for all $\ell \in \mathbb N$;
\item[(ii)] $\mathcal J(X^\ell,W^\ell) \geq \| X^\ell \|_*$, for all $\ell \in \mathbb N$;
\item[(iii)] There exists a constant $\mathscr A>0$ such that
\begin{equation}\label{aaa}
2 \mathscr A \left [ \mathcal J(X^\ell,W^\ell) - \mathcal J(X^{\ell+1},W^{\ell+1}) \right ] \geq \| X^{\ell+1} - X^\ell\|_F^2.
\end{equation}
This in particular implies
\begin{equation}\label{bbb}
\lim_{\ell \to \infty} \| X^{\ell+1} - X^\ell\|_F^2 = 0.
\end{equation}
\end{itemize}
\end{proposition}
\begin{proof}
(i) By the optimality properties of $X^{\ell+1}$ and $W^{\ell+1}$, along with the admissibility of $W^{\ell}$ in the optimization of $W^{\ell+1}$ (as $\varepsilon_{\ell+1} \leq \varepsilon_{\ell}$), we obtain the following chain of inequalities
$$
 \mathcal J(X^{\ell+1},W^{\ell+1}) \leq  \mathcal J(X^{\ell+1},W^{\ell}) \leq  \mathcal J(X^{\ell},W^{\ell}).
$$ 
(ii) 
A direct computation shows
$$
\mathcal J(X^\ell,W^\ell) = \left [ \sum_{i=1}^r \s^{\ell}_i + \sum_{i=r+1}^n \frac{(\s_i^\ell)^2 + \varepsilon_\ell^2}{2\varepsilon_{\ell}}  \right],
$$
where $r$ is the largest integer for which $\sigma_r^{\ell} \geq \varepsilon_{\ell}$.
Since $(\s_i^\ell)^2 + \varepsilon_\ell^2 \geq 2\varepsilon_{\ell} \cdot \s_i^{\ell}$, we have
$$
\mathcal J(X^\ell,W^\ell) \geq \| X^{\ell} \|_*.
$$
(iii) We have the following estimate,
\begin{align*}
&2 \left [ \mathcal J(X^\ell,W^\ell) - \mathcal J(X^{\ell+1},W^{\ell+1}) \right ] \geq  \left [ \mathcal J(X^\ell,W^\ell) - \mathcal J(X^{\ell+1},W^{\ell}) \right ] \\
&= \langle W^\ell X^\ell, X^\ell \rangle-  \langle W^\ell X^{\ell+1}, X^{\ell+1} \rangle\\
&=  \langle W^\ell (X^\ell +  X^{\ell+1}), X^\ell-  X^{\ell+1} \rangle  
+ 2 i \Im m (\langle W^\ell X^\ell, X^{\ell+1} \rangle).
\end{align*}
By optimality of $ X^{\ell+1}$ and in virtue of \eqref{optcond}, we can add $0=- 2  \langle W^\ell X^{\ell+1}, X^\ell-  X^{\ell+1} \rangle$ and $0=2 i \Im m (\langle W^\ell (X^{\ell+1}-X^\ell), X^{\ell+1} \rangle)$ (note also that $ X^\ell-  X^{\ell+1} \in \ker \A$), and we obtain
\begin{eqnarray*}
2 \left [ \mathcal J(X^\ell,W^\ell) - \mathcal J(X^{\ell+1},W^{\ell+1}) \right ] &\geq&  \langle W^\ell (X^\ell -  X^{\ell+1}), X^\ell-  X^{\ell+1} \rangle \\
&=&  \| (W^\ell)^{1/2} (X^\ell -  X^{\ell+1}) \|_F^2.
\end{eqnarray*}
On the one hand, 
$$
\| X^\ell -  X^{\ell+1} \|_F^2 \leq \| (W^\ell)^{1/2} (X^\ell -  X^{\ell+1}) \|_F^2 \| (W^\ell)^{-1/2}\|_F^2,
$$
and
$$
\| (W^\ell)^{-1/2}\|_F^2 =  ( n-r) \varepsilon_\ell^{1/2} + \sum_{i=1}^r (\s_i^\ell)^{1/2}.
$$
On the other hand, 
$$
\s_i^\ell \leq \| X^\ell\|_* \leq \mathcal J(X^1,W^0).
$$
Hence, there exists a constant $\mathscr A>0$ such that $\| (W^\ell)^{-1/2}\|_F^2 \leq \mathscr A$, for all $\ell \in \mathbb N$. It follows from \eqref{aaa} that
\[
\sum_{\ell = 0}^N \| X^{\ell+1} - X^{\ell} \|_F^2 \leq 2 \mathscr A \left[\mathcal J(X^{0},W^{0}) - \mathcal J(X^{N+1},W^{N+1})\right] \leq 2 \mathscr A \mathcal J(X^{0},W^{0}).
\]
This implies that the sum on the left converges as $N \to \infty$, and \eqref{bbb} follows.
\end{proof}

\subsection{Null space property}


We need to enforce additional conditions in order to obtain a simultaneous promotion of the approximation to both minimal rank and minimal nuclear norm solution. For this purpose we {\bnew recall} the rank null space property, {\bnew which is the analog} to the null space property for the vector case in compressed sensing \cite{codade09,grni03,fora10-1,ra10}.

\begin{definition}\label{def:RNSP} A map $\A : M_{n \times p} \to \mathbb{C}^m$ satisfies the \emph{rank null space property (RNSP)} of order $k$ if for all $H \in \ker \A \setminus \{0\}$ and all
decompositions $H = H_1 + H_2$ with $k$-rank $H_1$ it holds
\[
\| H_1 \|_* < \|H_2\|_*.
\]
\end{definition}

The rank null space property is closely linked to the nuclear norm minimization problem
\begin{equation}\label{nucl:min}
\min_{X \in M_{n \times p}} \|X\|_* \quad \mbox{ subject to } \quad \A (X) = \M.
\end{equation}
The following theorem has been shown in \cite[Theorem 3]{harexuXX}. 
\begin{theorem}\label{prop:RNSP} Let $\A: M_{n \times p} \to \mathbb{C}^m$ be a linear map. 
Then, every $k$-rank $X \in M_{n \times p}$ is the unique solution of the problem \eqref{nucl:min} 
for the datum $\M = \A(X)$ if and only if 
$\A$ satisfies the RNSP of order $k$. 
\end{theorem}

We also need a slightly stronger condition.

\begin{definition}\label{def:SRNSP} A map $\A : M_{n \times p} \to \mathbb{C}^m$ satisfies the \emph{strong rank null space property (SRNSP)}  of order $k$ with 
constant $\eta \in (0,1)$ if for all $X \in \ker \A \setminus \{0\}$ and all
decompositions $X = X_1 + X_2$ with $k$-rank $X_1$ there exists another
decomposition $X = H_1 + H_2$ such that $H_1$ is of rank at most $2k$, 
$\langle H_1,H_2\rangle = 0$, $X_1 H_2^* = 0$, $X_1^* H_2 = 0$, and
\[
\| H_1 \|_* \leq \eta \|H_2\|_*.
\]
\end{definition}
{\bnew
\begin{remark} The above definition is equivalent to \cite[Definition II.4]{mofa10}: Let $GL(U,V)$ be the subspace of $n \times p$ matrices of rank at most $k$, whose row space belongs to the span of $V \in M_{n \times k}$ 
and whose column space belongs to the span of $U \in M_{n \times k}$. Let $S_k = \{GL(U,V) : U \in M_{n \times k},V \in M_{p \times k}, U^*U = V^*V = I\}$. For $T \in GL(U,V)$, we let ${\mathcal P}_T$ be the projection onto $T$, and $T^\perp$ denotes the orthogonal complement of $T$ in $M_{n\times p}$. Then a map $\A : M_{n \times p} \to \mathbb{C}^m$ satisfies the \emph{strong rank null space property (SRNSP)}  of order $k$ with constant $\eta \in (0,1)$ if
\[
\|{\mathcal P}_T(X)\|_* \leq \eta \|{\mathcal P}_{T^\perp}\|_*
\]
for all $T \in S_k$ and $H \in \ker  \A \setminus \{0\}$.
\end{remark} 
}

Further, we introduce the {\it best $k$-rank
approximation error} in the nuclear norm,
\[
\rho_k(X)_* := \min_{Z \in M_{n \times p}: \rank(Z) \leq k} \|X-Z\|_*.
\]
It is well-known that the minimum is attained at the $k$-spectral truncation
$Z = X_{[k]}$.
The SRNSP implies the following crucial inequality.
\begin{lemma}[Inverse triangle inequality]
\label{lm7}
Assume that $\A: M_{n \times p} \to \mathbb C^m$ satisfies the SRNSP of order $k$ and
constant $\eta \in (0,1)$.
Let $X, Z$ be matrices such that
$$
\A(X) = \A(Z).
$$
Then we have the following inverse triangle inequality
\begin{equation}
\label{invtr2}
\| X - Z\|_* \leq \frac{1 + \eta}{1- \eta} \left ( \| Z\|_* - \| X \|_* + 2\rho_k(X)_* \right ).
\end{equation}
\end{lemma}
{\bnew This result was independently obtained in \cite[Lemma II.5]{mofa10}.
Inverse triangle inequalities are well-known in the context of the classical null space property for vectors \cite{dadefogu10}, see also the survey paper \cite{pink11}.}
\begin{proof} Let $X_0 := X_{[k]}$ be the best $k$-rank approximation of $X$ in the nuclear norm, and
set $X_c = X - X_0$.
Let $H = X- Z = X_0 + X_c - Z \in \ker \A$. Set $X_1 = Z - X_c$, so that $H = X_0 - X_1$. 
By the definition of the strong null space
property, there exists a decomposition $H = H_0 + H_c$ with
$\rank(H_0) \leq 2k$, $\langle H_0, H_c \rangle =0$, $X_0 H_c^* = 0$, $X_0^* H_c=0$ and such that $\|H_0\|_* \leq \eta \|H_c\|_*$. Hence, by Lemma \ref{lm5} recalled below, we have
\begin{eqnarray*}
\| X_0 - H_c \|_* = \| H_c \|_* + \|X_0\|_*
\end{eqnarray*}
and, as $H_0 + X_1 = X_0 - H_c$, we have equivalently
\begin{eqnarray*}
\| H_c \|_* + \|X_0\|_* = \|H_0 + X_1\|_* \leq \|H_0 \|_* + \| X_1\|_*.
\end{eqnarray*}
The inequality 
 $\|H_0\|_* \leq \eta \|H_c\|_*$ yields
\begin{eqnarray*}
\| H_c \|_* + \|X_0 \|_* \leq \| X_1\|_* + \eta \|H_c\|_*,
\end{eqnarray*}
or, equivalently,
$$
\| H_c\|_* \leq \frac{1}{1-\eta} \left ( \| X_1 \|_* - \| X_0\|_* \right).
$$
Now, using $X_1 = Z- X_c$, we obtain
\begin{align}
\| X- Z\|_* &= \| H_0 + H_c \|_* \leq (1 + \eta) \| H_c\|_* \leq \frac{1+ \eta}{1-\eta} ( \| X_1\|_* - \| X_0\|_* )\notag\\
&  \leq \frac{1+ \eta}{1-\eta} ( \| Z \|_* - \| X_0\|_* + \|X_c\|_*)
\leq \frac{1+ \eta}{1-\eta} ( \| Z \|_* - \| X - X_c\|_* + \|X_c\|_*)\notag\\
& \leq \frac{1+ \eta}{1-\eta} ( \| Z \|_* - \| X\|_* + 2 \|X_c\|_*).
\notag
\end{align}
The proof is completed by noting that $\|X_c\|_* = \|X-X_0\|_* = \rho_k(X)_*$.
\end{proof}

Note that if $X$ is $k$-rank then $\rho_k(X)_* = 0$ and \eqref{invtr2} reduces to
\begin{equation}
\label{invtr}
\| X - Z\|_* \leq \frac{1 + \eta}{1- \eta} \left ( \| X\|_* - \| Z \|_* \right).
\end{equation}
Now we easily conclude that the SRNSP implies stable low-rank recovery via 
nuclear norm minimization.
\begin{corollary}\label{cor:SRNMP} Suppose $\A : M_{n \times p} \to \mathbb{C}^m$ satisfies the strong null space property of order $k$ with some constant $\eta \in (0,1)$. Let $X \in M_{n \times p}$ and $\M = \A(X)$. Then
a solution $\bar X$ of \eqref{nucl:min} satisfies
\[
\|X - \bar X\|_* \leq 2 \frac{1+\eta}{1-\eta} \rho_k(X)_*.
\]
In particular, every $k$-rank $X \in M_{n \times p}$ is
the unique solution of the nuclear norm minimization problem \eqref{nucl:min}. Consequently,
the strong rank null space property of order $k$ with some constant $\eta < 1$ 
implies the rank null space property of order $k$. 
\end{corollary}
\begin{proof} 
Clearly $\|\bar X\|_* \leq \|X\|_*$ and
\eqref{invtr2} yields $\|X-\bar X\|_* \leq  {\nnew{2}}(1+\eta)/(1-\eta) \rho_k(X)_*$. Since $\rho_k(X)_* = 0$
for all $k$-rank $X$, the latter estimate implies exact recovery of such matrices.
Theorem 
\ref{prop:RNSP} implies also that $\A$ satisfies the rank null space property of order $k$.
\end{proof}

The SRNSP is somewhat difficult to analyze directly. As in the vector case \cite{fora10-1,ra10}, it is implied by the
(rank) restricted isometry property, Definition \ref{def:rip}:

\begin{proposition}\label{prop:RIP:RNSP} 
Assume that $\A:M_{n \times p} \to \mathbb{C}^m$ has restricted isometry constant 
\begin{equation}
\label{ripopt}
\delta_{4k} < \sqrt{2} - 1 \approx 0.41\;.
\end{equation}
Then $\A$ satisfies the SRNSP of order $k$ with constant 
\[
\eta = \sqrt{2} \frac{\delta_{4k}}{1-\delta_{3k}} \quad \in (0,1)\;.
\]
\end{proposition}

Our proof uses \cite[Lemma 3.3]{capl09} and \cite[Lemma 2.3]{fapareXX}, which we recall here for the reader's convenience.

\begin{lemma}[Lemma 3.3 in \cite{capl09}]\label{lem:RIP:aux} Let $\A : M_{n\times p} \to \mathbb{C}^m$ with restricted isometry constant $0<\delta_k$.
Then for two matrices $X,Y \in M_{n \times p}$ with $\rank(X) + \rank(Y) \leq k$ and $\langle X,Y\rangle = 0$
it holds
\[
|\langle \A(X), \A(Y) \rangle| \leq \delta_k \|X\|_F \|Y\|_F.
\]
\end{lemma}

\begin{lemma}[Lemma 2.3 in \cite{fapareXX}]
\label{lm5}
Let $X$ and $Z$ be matrices of the same dimensions. If $X Z^* = 0$ and $X^* Z =0$ then 
$$
\| X + Z\|_* = \| X\|_* +\| Z\|_*.
$$ 
\end{lemma}

\begin{proof}[of Proposition \ref{prop:RIP:RNSP}] For convenience 
we give a simple proof, which results in the stated non-optimal constant in \eqref{ripopt}, along
the lines of the proof of Theorem 2.6 in \cite{ra10}. Let $X \in \ker \A \setminus \{0\}$ 
and $X_1$ be of rank at most $k$, so that, by setting $X_2 = X - X_1$, we have $X = X_1 + X_2$.
Let 
\[
X_1 = U \left( \begin{matrix} \Sigma & 0 \\ 0 & 0 \end{matrix} \right) V^*
\]
be the (full) singular value decomposition of $X_1$ with $\Sigma$ being 
a $k \times k$ diagonal matrix. Set $\hat{H} = U^* X V$ and partition
\[
\hat{H} = \left( \begin{matrix} \hat{H}_{11} & \hat{H}_{12} \\ \hat{H}_{21} & \hat{H}_{22} \end{matrix} \right),
\]
where $\hat{H}_{11} \in M_{k \times k}$. Now, set
\[
H_0 = U \left( \begin{matrix} \hat{H}_{11} & \hat{H}_{12} \\ \hat{H}_{21} & 0 \end{matrix} \right) V^*,
\quad H_c = U \left(\begin{matrix} 0 & 0 \\ 0 & \hat{H}_{22} \end{matrix} \right) V^*.
\]
Then $\rank(H_0) \leq 2k$, $X_1  H_c^*= 0$, $X_1^* H_c = 0$, and $\langle H_0, H_c \rangle = 0$ as desired. It remains to prove that $\|H_0\|_* \leq \eta \|H_c\|_*$. Let $\hat{H}_{22}$
have singular value decomposition 
$\hat{H}_{22} = \tilde{U} \diag(\tilde{\sigma}) \tilde{V}^*$, in particular, the entries of 
$\tilde{\sigma}$ are in decreasing order. Now we choose an integer $\ell > 0$ and decompose
the vector $\tilde{\sigma}$ into vectors $\sigma^{(j)}$, $j=1,2,\dots$, of support length $\ell$ defined by
\[
\sigma^{(j)}_i = \left\{ \begin{matrix} \tilde{\sigma}_i,\;\mbox{ if } \ell (j-1) < i \leq \ell j,\\ 
0,\; \mbox{ otherwise }.\end{matrix} \right.
\]
Clearly, $\sum_j \sigma^{(j)} = \tilde{\sigma}$. We set, for $j=1,2,\dots$,
\[
H_{j} = U \left(\begin{matrix} 0 & 0 \\ 0 & \tilde{U} \diag(\sigma^{(j})  \tilde{V}^* \end{matrix} \right) V^*.
\]
Then $\rank(H_j) \leq \ell$, $H_c = \sum_{j\geq 1} H_j$ and $\langle H_0,H_j\rangle = 0$ for
all $j=1,2,\dots$. Also, observe that since $H \in \ker \A$ we have 
$\A(H_0+H_1) = \sum_{j\geq 2} \A(-H_j)$. Now we first use that 
$\|H_0\|_* \leq \sqrt{2k} \|H_0\|_F \leq \sqrt{2k} \|H_0 + H_1\|$ by orthogonality of $H_0$ and $H_1$, 
and estimate the latter. To this end, we use orthogonality of $H_0$ and $H_1$, the RIP and Lemma \ref{lem:RIP:aux},
\begin{align}
\|H_0 + H_1\|_F^2 &\leq \frac{1}{1-\delta_{2k+\ell}} \|\A(H_0+H_1)\|^2_F = \frac{1}{1-\delta_{2k+\ell}}
\langle \A(H_0+H_1), \sum_{j\geq 2} \A(H_j) \rangle \notag\\
& =   \frac{1}{1-\delta_{2k+\ell}} \sum_{j\geq 2} \langle \A(H_0+H_1), \A(H_j) \rangle
\leq  \frac{\delta_{2k+2\ell}}{1-\delta_{2k+\ell}} \|H_0+H_1\|_F \sum_{j\geq 2} \|H_j\|_F.\notag
\end{align}
Hence, $\|H_0+H_1\|_F \leq  \frac{\delta_{2k+2\ell}}{1-\delta_{2k+\ell}} \sum_{j\geq 2} \|H_j\|_F$. Since the sequence $\tilde{\sigma}$ is nonincreasing, we have by definition of $\sigma^{(j)}$ that, for $j \geq 2$ and for all $i$,
$\sigma^{(j+1)}_i \leq \frac{1}{\ell} \|\sigma^{(j)}\|_{\ell_1} = \frac{1}{\ell} \|H_{j}\|_*$, 
and therefore,
\[
\|H_{j+1}\|_F = \left(\sum_{i=j\ell + 1}^{(j+1)\ell} \sigma_i^2\right)^{1/2} \leq \frac{1}{\sqrt{\ell}} \|H_{j}\|_*.
\]
Combining our estimates
we obtain
\begin{align}
\|H_0\|_* &\leq \sqrt{2k} \|H_0+H_1\|_F \leq  \frac{\delta_{2k+2\ell}}{1-\delta_{2k+\ell}} \sqrt{2k} \sum_{j\geq 2} \|H_j\|_F
\leq  \frac{\delta_{2k+2\ell}}{1-\delta_{2k+\ell}} \sqrt{\frac{2k}{\ell}} 
\sum_{j\geq 1} \|H_j\|_* \notag\\
&= \frac{\delta_{2k+2\ell}}{1-\delta_{2k+\ell}} \sqrt{\frac{2k}{\ell}} \|H_c\|_*.
\label{choose:ell}
\end{align}
In the last step we applied Lemma \ref{lm5}. 
Now we choose $\ell = k$ and obtain
$
\| H_0 \|_* \leq \eta \|H_c\|_*
$
with 
\[
\eta = \sqrt{2} \frac{\delta_{4k}}{1-\delta_{3k}}.
\]
Since $\delta_{3k} \leq \delta_{4k}$, we have $\eta < 1$ if $\delta_{4k} < \sqrt{2} -1$.
\end{proof}

We note that different choices of $\ell$ in \eqref{choose:ell} (say $\ell = 2k$, or $\ell = \lceil k/2\rceil$) lead to slightly different conditions. Also, we do not claim optimality of the constant
$\sqrt{2}-1$. We expect that with a more complicated proof as in \cite{capl09} or 
\cite{fo09,fola09,limo11}, one can still improve this value. Our goal here was to rather provide 
a simple proof. {\bnew After submission of an initial version of our paper, we became aware of \cite{oymofaha11}, where 
an elegant 
relationship between vector and matrix cases is shown. 
In particular, properties and recovery guarantees valid for sparse vector recovery by means of $\ell_1$-minimization are shown to imply corresponding 
properties and recovery guarantees for low-rank matrices by means of nuclear-norm minimization.} {\nnew It seems that together with \cite{limo11}
this implies slightly better guarantees than provided in  Proposition \ref{prop:RIP:RNSP}.}

\subsection{Analysis of the algorithm when $\A$ satisfies the null space property}

Before presenting the main result of this paper, we introduce another functional, dependent on a parameter $\varepsilon > 0$,
\begin{equation}\label{def:Jepsilon}
\mathcal{J}_{\varepsilon}(X) := J_{\varepsilon}(\Sigma(X)) = \sum_{i=1}^{n} j^{\varepsilon}(\sigma_i(X)),
\end{equation}
where 
$$
j^{\varepsilon}(u) = \left\{ \begin{array}{cc} |u|, & |u| \geq \varepsilon, \nonumber \\
\frac{u^2 + \varepsilon^2}{2\varepsilon}, & u < \varepsilon.
\end{array} \right. 
$$

\begin{theorem}
\label{secondres}
Consider the IRLS-M algorithm for low-rank matrix recovery, with parameters $\gamma = 1/n$ and $K \in  \mathbb{N}$.  
Let $\A: M_{n \times p} \rightarrow \mathbb{C}^m$ be a surjective map.
Then, for each set of measurements $\M \in \mathbb{C}^m$, the sequence of matrices $(X^{\ell})_{\ell \in \mathbb{N}}$ produced by the IRLS-M algorithm has the following properties:
\begin{itemize}
\item[(i)] Assume that $\A$ satisfies the strong rank null space property of order $K$ (Definition \ref{def:SRNSP}). 
If $\lim_{\ell \to \infty} \varepsilon_\ell =0$, then the sequence $(X^{\ell})_{\ell \in \mathbb{N}}$ converges to a $K$-rank matrix $\bar{X}$ agreeing with the measurements; in this case, $\bar{X}$ is also the unique nuclear norm minimizer.  
\item[(ii)] If $\lim_{\ell \to \infty} \varepsilon_{\ell} = \varepsilon > 0$ then every subsequence of $X^{\ell}$ has a convergent subsequence. Each
accumulation point $\tilde X$ of $(X^\ell)_{\ell\geq 1}$ coincides with a minimizer $\bar X$ of the functional $\mathcal{J}_{\varepsilon}$ subject to $\A(\bar X) = \M$. In particular, if this minimizer is unique
then the full sequence $(X^\ell)_{\ell\geq 1}$ converges to it. 

If, in addition, $\A$ satisfies the strong rank null space property of order $K$ (Definition \ref{def:SRNSP}) 
and constant $\eta < 1 - \frac{2}{K-2}$, 
then each accumulation point $\tilde{X}= \bar X$ of $(X^\ell)_{\ell \geq 1}$ 
satisfies, for any matrix $X$ such that $\A (X) = \M$, and for $k < K - \frac{2\eta}{1-\eta}$, 
$$
\| X - \bar X\|_* 
\leq \Lambda \rho_k(X)_*,
$$
where 
\[
\Lambda := \frac{4(1+\eta)^2}{(1-\eta)^2((K-k)(1-\eta) - 2\eta)} + \frac{2(1+\eta)}{1-\eta}.
\]
\item[(iii)] In particular, if $\A$ satisfies the strong rank null space property of order $K$ with constant $\eta < 1 - \frac{2}{K-2}$, and if
there exists a $k$-rank matrix $X$ satisfying $\A (X) = \M$, then necessarily $\varepsilon = 0$.
\end{itemize}
\end{theorem}

\begin{proof} 
Note that since $0 \leq \varepsilon_{\ell+1} \leq \varepsilon_\ell$, the sequence $(\varepsilon_\ell)_{\ell \in \mathbb N}$ always converges.  

(i) 
If $\varepsilon_\ell =0$ for some $\ell$, then we set $\bar{X} = X^\ell$. If $\varepsilon_\ell >0$ for all $\ell$, then there exists a subsequence, such that
$$
\varepsilon_{\ell_j+1} <\varepsilon_{\ell_j}, \quad j \in \mathbb N, 
$$
and $\varepsilon_{\ell_j+1} = \gamma \s_{K+1}(X^{\ell_j+1})$. Because of Proposition \ref{prop2} (i) and (ii), $(X^{\ell_j+1})_j$ is bounded. Hence, we can extract a further subsequence, denoted again by $(X^{\ell_j+1})_j$ for simplicity, 
which converges to some $\tilde{X}  = \lim_{j \to \infty} X^{\ell_j}$. Since $\lim_{j \to \infty} \varepsilon_{\ell_j+1} =0$, 
an application of Theorem \ref{we3} 
immediately yields
$\sigma_{K+1}(\tilde{X}) = 0$, and  $\tilde{X}$ is a $K$-rank solution of $\A (X) = \M$. 

The strong rank null space property together with Corollary \ref{cor:SRNMP} 
implies that $\bar{X}$ is the unique nuclear norm minimizer $\bar X$. Next we show that the whole sequence $(X^\ell)_\ell$ converges to $\bar{X}$. We have 
$$
\mathcal J(X^\ell,W^\ell) = \left [ \sum_{i=1}^{r(\ell)} \s^{\ell}_i + \sum_{i=r(\ell)+1}^n \frac{(\s_i^\ell)^2 + \varepsilon_\ell^2}{2\varepsilon_{\ell}}  \right],
$$
where we recall that $r(\ell)$ is the smallest integer for which $\sigma_i^{\ell} < \varepsilon_{\ell}$ if $i > r(\ell)$.   Since $\mathcal J(X^\ell,W^\ell,\varepsilon_\ell)$ is a monotonically decreasing sequence and $\mathcal J(X^{\ell_j},W^{\ell_j},\varepsilon_{\ell_j}) \to \| \bar{X}\|_*$ for $j \to \infty$, we also have $\mathcal J(X^\ell,W^\ell,\varepsilon_\ell) \to  \| \bar{X}\|_*$ for $\ell \to \infty$.
We have the estimates from above and below
$$
\mathcal J(X^\ell,W^\ell,\varepsilon_\ell) - \sum_{i=r(\ell)+1}^n \Big( \frac{(\s_i^\ell)^2 + \varepsilon_\ell^2}{2\varepsilon_{\ell}} - \sigma_i^{\ell} \Big) \leq \| X^\ell\|_* \leq \mathcal J(X^\ell,W^\ell,\varepsilon_\ell) .
$$
As $\ell \to \infty$, $\sigma_i^{\ell} < \varepsilon_{\ell} \rightarrow 0$ for $i \geq r(\ell) + 1$, and we compute
$$
\lim_{\ell \rightarrow \infty} \mathcal \sum_{i=r(\ell)+1}^n \Big( \frac{(\s_i^\ell)^2 + \varepsilon_\ell^2}{2\varepsilon_{\ell}} - \sigma_i^{\ell} \Big) \leq \lim_{\ell \rightarrow \infty} \sum_{i=r(\ell)+1}^n \Big( \frac{\varepsilon_\ell - \s_i^{\ell}}{2} \Big) \leq \lim_{\ell \rightarrow \infty}  n\varepsilon_\ell = 0,
$$
and we obtain
$$
\lim_{\ell \to \infty} \| X^\ell\|_* = \| \bar{X}\|_*.
$$
Since $\bar{X}$ is a $K$-rank matrix, we can apply the inverse triangle inequality \eqref{invtr} to obtain
\begin{equation}
\label{estim}
\| X^{\ell} - \bar X \|_* \leq \frac{1+\eta}{1- \eta} \left ( \| X^\ell\|_* - \| \bar X\|_* \right ),
\end{equation}
where $\eta < 1$ by assumption. Taking the limit for $\ell \to \infty$ in \eqref{estim}, we obtain $\lim_{\ell \to \infty} X^\ell = \bar{X}$.

(ii) In the case that $\lim_{\ell \to \infty} \varepsilon_\ell =\varepsilon >0$ we use the functional $\J_\varepsilon$ introduced in
\eqref{def:Jepsilon}.
Observe that ${J}_\varepsilon: \mathbb R^n \to \mathbb R_+$ is differentiable, convex, 
and {\it absolutely symmetric}, 
see Section \ref{appendix:sing}. By Proposition \ref{diff:singvalues}, we therefore have, for $X$ with singular value
decomposition $X = U \diag(\sigma(X)) V^*$,
\[
\nabla \J_\varepsilon (X) = U \diag(j_\varepsilon'(\sigma_i(X))) V^*,
\] 
where the derivative of $j_\varepsilon$ is given by
\[
j_{\varepsilon}'(u) = \left\{ \begin{array}{cc} \sgn(u), & |u| \geq \varepsilon, \nonumber \\
u/\varepsilon, & |u| < \varepsilon.
\end{array} \right. 
\]
We have
$$
\bar{X}= \bar{X}(\varepsilon) \in \arg \min_{\A(X) = \M} \mathcal J_\varepsilon(X)
$$
if and only if $\langle \nabla \J_\varepsilon (\bar{X}), H \rangle = 0$ for all $H \in \ker \A$. It is now straightforward to verify that the latter is
equivalent to
\begin{equation}
\label{opteps}
\langle W \bar{X}, H \rangle =0, \mbox{ for all } H \in \ker \A,
\end{equation}
where $W = [((\bar{X}\bar{X}^*)^{1/2})_{\varepsilon}]^{-1}$. Here, we recall that $Z_{\varepsilon}$ is the 
$\varepsilon$-regularization of $Z$, see \eqref{estab}.

Since $X^\ell$ is a bounded sequence, it has accumulation points. Let us first show that any accumulation point of $X^\ell$ is a minimizer of ${\cal J}_{\varepsilon}$ subject to the constraint $\A(X) = \M$. Note that  ${\cal J}_{\varepsilon}$ is not \emph{strictly} convex, so such a minimizer is not necessarily unique.  Let $(X^{\ell_j})_{j \in \mathbb N}$ be any convergent subsequence and 
$\tilde X$ its limit. 
Note that $W^{\ell_j} = [((X^{\ell_j}(X^{\ell_j})^*)^{1/2})_{\varepsilon_{\ell_j}}]^{-1}$. Hence, $W^{\ell_j}$ depends continuously on 
$X^{\ell_j}$ so that also $W^{\ell_j}$ converges to a limit 
$\tilde W = \lim_{j\to \infty} W^{\ell_j} = \big[ ( (\tilde{X} \tilde X^*)^{1/2})_{\varepsilon} \big]^{-1}$.
By invoking Proposition \ref{prop2} (iii) we have also $X^{\ell_j+1} \to \tilde X$ for $j \to \infty$, hence by \eqref{optcond}
$$
\langle \tilde W \tilde X, H \rangle = \lim_{j \to \infty} \langle W^{\ell_j}  X^{\ell_j+1}, H \rangle =0, \mbox{ for all } H \in \ker \A.
$$
This is exactly the optimality condition \eqref{opteps} and therefore $\tilde X =\bar X$ 
is a minimizer of 
${\cal J}_{\varepsilon}(X)$.

We now prove the error estimate.  To begin, let $X^{\ell_j}$ be a converging subsequence of $X^\ell$ with limit $\tilde{X}$, which, for simplicity we denote again by
 $X^\ell$. As just outlined $\tilde{X}=\bar X$ is a minimizer of $\mathcal{J}_{\varepsilon}$ subject to $\A (\tilde X) = \M$. 
Note that for any matrix $X$ satisfying the constraint 
$\A (X) = \M$, and for any minimizer $\bar X$ of $\mathcal{J}_{\varepsilon}$ subject to the constraint $\A(\bar X) = \M$, 
 $$
 \| \bar X\|_{*} \leq {\cal{J}}_{\varepsilon}(\bar X) \leq {\cal{J}}_{\varepsilon}(X) \leq \| X \|_{*} + n \varepsilon,
 $$
where we have used the fact that $\bar{X}$ is a minimizer of $\mathcal{J}_{\varepsilon}$, so that
 \begin{equation}\label{Xeps:min}
 \| \bar X \|_{*} -  \| X \|_{*} \leq n \varepsilon.
 \end{equation}
By the inverse triangle inequality, Lemma \ref{lm7}, and \eqref{Xeps:min}
\begin{align}
\label{nsp:1}
\| \bar{X}- X \|_{*} 
&\leq \frac{1 + \eta}{1 - \eta} ( \| \bar{X}\|_{*} - \| X \|_{*}+ 2 \rho_k(X)_* )
\leq \frac{1 + \eta}{1 - \eta} ( n\varepsilon  + 2 \rho_k(X)_*). 
\end{align}
Thus we reach
 $$
 n \varepsilon = \lim_{\ell \rightarrow \infty} n \varepsilon_{\ell} \leq \lim_{\ell \rightarrow \infty} \sigma_{K+1} (X^{\ell}) 
 = \sigma_{K+1}(\bar{X}).
 $$
It follows from Proposition \ref{rearrangers} that 
\begin{align}
(K + 1 -k)  n\varepsilon & \leq (K+1-k) \sigma_{K+1}(\bar{X}) \leq \|\bar{X}-X\|_* + \rho_{k}(X)_* \notag\\
& \leq \frac{1 + \eta}{1 - \eta} ( n \varepsilon + 2\rho_{k}(X)_*) + \rho_k(X)_*. \notag
\end{align}
Under the assumption that $K - k > \frac{2\eta}{1 - \eta}$, this inequality yields
$$
n \varepsilon \leq \frac{4(1+\eta)}{(1-\eta)((K-k)(1-\eta)-2\eta)} \rho_k(X)_*.
$$
Plugging this back into \eqref{nsp:1}, we arrive at the desired result.
\end{proof}

{\bnew \subsection{Discussion of related work}

While we were finishing this paper,  Maryam Fazel informed us of her joint work with Karthik Mohan, where a similar iteratively least squares minimization algorithm is studied \cite{mofa10}. We shortly outline the differences between our contribution and \cite{mofa10}. In \cite{mofa10} a direct and faithful generalization of the algorithm analyzed in 
\cite{dadefogu10} for sparse vector recovery is provided.
In particular, instead of \eqref{wn}
$$
W^{\ell}= U^\ell \big( \diag(\max\{ \sigma^\ell_j, \varepsilon_\ell\})\big)^{-1} (U^\ell)^*,
$$
the update rule for the weights is given there by
$$
W^{\ell}= U^{\ell} \diag (((\sigma^\ell_j)^2 + \varepsilon_\ell^2)^{-1/2}) (U^\ell)^*.
$$
On the one hand, the drawback of our up-date rule \eqref{wn} is that is it not anymore equivalent to an unconstrained minimization of the energy
$$
\widetilde{\mathcal J}(X^\ell,W) = \mathcal J(X^\ell,W) + \frac{\varepsilon_\ell}{2} \| W^{1/2}\|_F^2, 
$$
with respect to the second variable $W \succ 0$, but to the constrained minimization of $\mathcal J(X^\ell,\cdot)$ as indicated in \eqref{wn2}.
Such equivalence had to be proven in Lemma \ref{lm3} and Proposition \ref{lm3'}, representing an additional technical difficulty 
in order to develop a proof of convergence, and introducing an element of novelty with respect to \cite{dadefogu10}. On the other hand, our choice is motivated by a significant 
improvement of the complexity of the algorithm when the operator $\A$ is {\it separable}, i.e., acts columnwise, 
$$
\A(X) = (\A_1 X_1,\dots, \A_p X_p) = (\M_1, \dots, \M_p) = \M \in \mathbb C^m,
$$
where $X_i$, $i=1,\dots,p$,  are the columns of $X$, and $\A_i$, $i=1,\dots, p$ are suitable matrices acting on the vectors $X_i$. 
We repeat that this case includes relevant situations such as the matrix completion problem. As clarified in Section \ref{woodsec}, the Woodbury formula, which can applied to the matrices $\A_i (W^\ell)^{-1} \A_i^*$, thanks to the particular structure
of the matrices $W^\ell$ as in \eqref{wn}, allows for their inversion with a much more economical cost than it would be possible by using the up-date rule proposed in  \cite{mofa10}. Moreover, while the analysis in  \cite{mofa10} also partially addresses reweighted least squares algorithms for Schatten $q$-norms, for $q<1$, in the case of the nuclear norm ($q=1$) the authors prove in \cite[Theorem II.7]{mofa10} only an analogous version of our Theorem \ref{secondres} (i), and do not develop a counterpart to (ii) and (iii).
}
\section{Appendix}

\subsection{Basic properties of singular values}

In \cite{we12} Weyl proved the following stability estimate on the singular values.
\begin{theorem}
\label{we3}
Let $X,Y \in M_{n \times p}$ be fixed. Then
$$
|\sigma_i(X)- \sigma_i(Y)| \leq \| X-Y\|_F, \quad i=1,2,\dots.
$$
\end{theorem}

\begin{proposition}
\label{rearrangers}
If $X \in M_{n \times p}$ and $Y \in M_{n \times p}$, then for any $j$, we have
$$
| \| X - X_{[j]} \|_{*} - \| Y - Y_{[j]} \|_{*} | \leq \| X - Y \|_{*},
$$
and for any $J > j$, we have
$
(J - j) \sigma_{J}(X) \leq \| X - Y \|_{*} + \| Y - Y_{[j]} \|_{*}
$.
\end{proposition}

\begin{proof}
The first statement is a simple chain of inequalities:
$$\| X - X_{[j]} \|_{*} \leq \| X - Y_{[j]} \|_{*} \leq \| X - Y \|_{*} + \| Y - Y_{[j]} \|_{*}.$$
The result follows as we may reverse the roles of $X$ and $Y$.
For the second inequality, it suffices to note that 
$(J-j) \sigma_J(X) \leq \| X - X_{[j]} \|_{*}$.
\end{proof}

\subsection{Differentiation and singular values}
\label{appendix:sing}

Let $\PP_n$ denote the set of all $n \times n$ permutation matrices.
A function $f : \R^n \to \R$ is called absolutely symmetric if 
\[
f(P D x) = f(x) \quad \mbox{for all } x \in \R^n, P \in \PP_n, 
\] 
and for all diagonal matrices $D$ having only the values $+1,-1$ on the diagonal. A function
$F : M_{n \times p} \to \R^n$ is called unitarily invariant if $F(U^* X V)$ for all $X \in M_{n \times p}$, and for all unitary
matrices $U \in M_{n \times n},V \in M_{p \times p}$. 
The following
fact is well-known, see for instance \cite[Proposition 5.1]{lese05}.
\begin{proposition}\label{prop:unitary:invariance} A function  $F : M_{n \times p} \to \R$, $n \leq p$, 
is unitarily invariant if and only if
$F(X) = f(\sigma(X))$ for  some absolutely symmetric function 
$f : \R^n \to \R$. Here, $\sigma(X) \in \R^n$ denotes the vector of singular values of $X$.
\end{proposition}

Let $f : \R^n \to \R$ be absolutely symmetric. 
Then the function $F(X) = f(\sigma(X))$ is convex
if and only if $f$ is convex. We have the following result concerning differentiation of $F$ \cite[Proposition 6.2]{lese05}.

\begin{proposition}\label{diff:singvalues} Let $f : \R^n \to \R$ be an absolutely symmetric and convex function. The function 
$F = f \circ \sigma$ on $M_{n \times p}$ is differentiable at $X \in M_{n\times p}$ if and only if $f$ is differentiable 
at $\sigma(X)$. In this case
\[
\nabla F(X) = U \diag( \nabla f (\sigma(X)) V^*, 
\]
where $X = U \diag(\sigma(X)) V^*$ for unitary matrices $U,V$.
\end{proposition}

\subsubsection*{Acknowledgments} We would like to thank  Ingrid Daubechies and Maryam Fazel for various
conversations on the topic of this paper.  Massimo Fornasier further acknowledges the financial support provided by the START-Prize ``Sparse Approximation and Optimization in High Dimensions'' of the Fonds zur F\"orderung der wissenschaftlichen Forschung (FWF, Austrian Science Foundation). The results of the paper also contribute to the project WWTF Five senses-Call 2006, Mathematical Methods for Image Analysis and Processing in the Visual Arts.  Holger Rauhut would like to thank the Hausdorff Center for Mathematics for generous support and excellent working conditions, and acknowledges funding through the WWTF project SPORTS (MA07-004).
Rachel Ward was funded in part by the National Science Foundation Postdoctoral Research Fellowship.

\bibliography{IRLSM}
\bibliographystyle{abbrv}
\end{document}